\documentclass[12pt, a4paper]{amsart}
\usepackage{amsmath}
\usepackage{geometry,amsthm,graphics,tabularx,amssymb}
\usepackage{amscd}
\usepackage{mathrsfs}
\newcommand{\BA}{{\mathbb{A}}}

\newcommand{\BC}{{\mathbb {C}}}

\newcommand{\BQ}{{\mathbb {Q}}}
\newcommand{\BR}{{\mathbb {R}}}

\newcommand{\BZ}{{\mathbb {Z}}}

\newcommand{\CO}{{\mathcal {O}}}

\DeclareMathOperator{\absNorm}{\mathfrak{N}}

\newcommand{\Aut}{{\mathrm{Aut}}}

\newcommand{\End}{{\mathrm{End}}}

\newcommand{\Gal}{{\mathrm{Gal}}}
\newcommand{\GL}{{\mathrm{GL}}}

\newcommand{\Gm}{{\mathbb{G}_m}}

\newcommand{\Hom}{{\mathrm{Hom}}}

\newcommand{\Ker}{{\mathrm{Ker}}}

\newcommand{\rank}{{\mathrm{rank}}}

\newcommand{\res}{{\mathrm{res}}}

\newcommand{\SL}{{\mathrm{SL}}}

\newcommand{\GO}{{\mathrm{GO}}}

\newcommand{\GU}{{\mathrm{GU}}}

\newcommand{\norm}[1]{\|{#1}\|}

\newcommand{\ad}{\operatorname{ad}}
\newcommand{\der}{\operatorname{der}}

\newcommand{\diag}{\operatorname{diag}}

\newcommand{\sgn}{\operatorname{sgn}}

\newcommand{\oH}{\operatorname{H}}
\newcommand{\oO}{\operatorname{O}}

\newcommand{\oU}{\operatorname{U}}

\newcommand{\g}{\mathfrak g}

\renewcommand{\k}{\mathfrak k}
\newcommand{\h}{\mathfrak h}

\renewcommand{\c}{\mathfrak c}

\newcommand{\f}{\mathfrak f}

\newcommand{\s}{\mathfrak s}

\newcommand{\fN}{\mathfrak N}

\newcommand{\Z}{\mathbb{Z}}
\newcommand{\C}{\mathbb{C}}
\newcommand{\R}{\mathbb R}

\newcommand{\A}{\mathbb{A}}

\newcommand{\OO}{\mathcal O}

\newcommand{\abs}[1]{\lvert#1\rvert}

\newcommand{\be}{\begin {equation}}
\newcommand{\ee}{\end {equation}}
\newcommand{\bee}{\begin {equation*}}
\newcommand{\eee}{\end {equation*}}

\theoremstyle{Theorem}

\theoremstyle{Theorem}
\newtheorem{introconjecture}{Conjecture}

\newtheorem{introtheorem}[introconjecture]{Theorem}

\newtheorem{thm}{Theorem}[section]

\theoremstyle{Theorem}
\newtheorem{lem}{Lemma}[section]

\theoremstyle{Theorem}
\newtheorem{prp}{Proposition}[section]
\newtheorem{corp}[prp]{Corollary}

\theoremstyle{Plain}

\theoremstyle{Definition}

\begin{document}

\title[$p$-adic $L$-functions]{$p$-adic $L$-functions for Rankin-Selberg convolutions over number fields}

\author{Fabian Januszewski}

\address{Institut f\"ur Algebra und Geometrie, Fakult\"at f\"ur Mathematik,
Karlsruher Institut f\"ur Technologie (KIT), Germany} \email{januszewski@kit.edu}


%\thanks{}

\begin{abstract}
We unconditionally construct cyclotomic $p$-adic $L$-functions for Rankin-Selberg convolutions for $\GL(n+1)\times \GL(n)$ over arbitrary number fields, and show that they satisfy an expected functional equation.
\end{abstract}

\maketitle

\tableofcontents

\section{Introduction}\label{s1}

In modern number theory $p$-adic $L$-functions play a central role, as they tend to reflect arithmetic properties more directly than their complex counterparts. Since the discovery of the $p$-adic Riemann $\zeta$-function by Kubota and Leopoldt in the 1960's, their deep arithmetic significance became apparent in Iwasawa's work on cyclotomic fields, which culminated in Mazur-Wiles' proof of the cyclotomic Iwasawa main conjecture.

Since then $p$-adic methods are omnipresent in modern number theory, and the construction of $p$-adic $L$-functions has been an important and in general unsolved problem. Even worse, being an arithmetic invariant, all we can do is construct $p$-adic $L$-functions for automorphic representations. Due to the lack of an Euler product in the $p$-adic world, we do not have a direct construction for motives.

%The known constructions may be grouped in three different classes. The first proceeds by explicit computation of the critical $L$-values, which seems only possible for $L$-series attached to Dirichlet characters, as excecuted by Euler in his infamous solution of the Basel problem, and reinterpreted by Kubota and Leopoldt, thereby essentially generalizing the Kummer congruences. The second approach, due to Serre, interprets the corresponding $L$-values as constant terms of Eisenstein series and uses $p$-adic interpolation of the latter to deduce the desired $p$-adic properties of constant terms. Recently there have been efforts to use the Langlands-Shahidi method in a $p$-adic context to the construction of $p$-adic $L$-functions. As of now it is not clear if this can eventually yield new instances. The third construction, pioneered by Mazur, Swinnerton-Dyer and Manin is based on modular symbols.

%Each of these methods comes with its own merits and defects. Serre's approach needs the existence of Eisenstein series and also seems to require an underlying Shimura variety, unavailable in many contexts.

A very successful method for constructing $p$-adic $L$-functions is via Mazur's modular symbols. They exist in purely topological contexts and behave very well, but their non-vanishing is an unsolved problem in general.

In this article we study a modular symbol on $\GL(n+1)\times \GL(n)$ computing the special values of the twisted Rankin-Selberg $L$-functions $L(s,(\pi\times\sigma)\otimes\chi)$ for regular algebraic cuspidal automorphic representations $\pi$ and $\sigma$ on $\GL(n+1)$ resp.\ $\GL(n)$. We allow an arbitrary number field as the base field. Thanks to results of Sun \cite{Sun}, the non-vanishing hypothesis in our setting has been proved. In particular our results are unconditional. Finally we apply this to the construction of $p$-adic $L$-functions and show that they always exist in the finite slope situation, and also satisfy a functional equation.

By recent results of Harris-Lan-Taylor-Thorne \cite{HLTT2013} and independently of Scholze \cite{scholze2013}, we know that whenever $k$ is totally real or a CM field, that we may associate to $\pi$ and $\sigma$ compatible $\ell$-adic systems of Galois representations $\rho_{\pi,\ell}$ resp.\ $\rho_{\sigma,\ell}$ with the property that
\begin{equation}\label{eq:slid}
L(s,(\pi\times\sigma)\otimes\chi)\;=_S\;
L(s-\frac{1-n(n+1)}{2},(\rho_{\pi,\ell}\otimes\rho_{\sigma,\ell})_\ell\otimes\chi),
\end{equation}
where the notation $=_S$ suggests that this identity holds for all but finitely many Euler factors. So far it is yet unkown if these Galois representations are de Rham at $\ell=p$ or even motivic, but this is to be expected. In the same spirit \eqref{eq:slid} is expected to be a strict identity. Despite these open questions the existence of these Galois representations gives a strong hint that our construction eventually bears arithmetic significance.

Our main Theorem is the following. Let $k$ be a number field and let $p$ be a rational prime. Let $\pi$ and $\sigma$ be regular algebraic irreducible cuspidal automorphic representations of $\GL_{n+1}(\BA_k)$ and $\GL_{n}(\BA_k)$, respectively. Then it is well known that the finite components $\pi^{(\infty)}$ and $\sigma^{(\infty)}$ are defined over a number field $\BQ(\pi,\sigma)$ (cf.\ \cite[Th\'eor\`eme 3.13]{Clo}; this statement extends to the global represenations in an appropriate sense \cite{Janpre}).
Define the compact abelian $p$-adic Lie group 
$$
  C(p^\infty) \;:=\;
\varprojlim_n 
k^\times\backslash\BA_k^\times/
(k\otimes_\BQ\BR)^0
(1+p^n\CO_k\otimes_\BZ\BZ_p)
\prod\limits_{v\nmid p\infty}\CO_{k,v}^\times
$$
which by class field theory is naturally isomorphic to the Galois group $\Gal(k_{p^\infty}/k)$ of the maximal abelian extension $k_{p^\infty}$ of $k$ unramified outside $p\infty$. Then, assuming that the weights of $\pi$ and $\sigma$ are compatible, and that $\pi$ and $\sigma$ are $p$-ordinary, or more generally of finite slope at $p$, we have
\begin{introtheorem}
For each $s_{\rm crit}=\frac{1}{2}+j$  ($j\in \Z$) critical for $L(s,\pi\times\sigma)$ there exist periods $\Omega_j^{\pm}\in\BC^\times$ and a $\BQ(\pi,\sigma)$-valued $p$-adic measure $\mu_j$ (a locally analytic distribution in the finite slope case) on $C(p^\infty)$, such that for all finite order characters $\chi:C(p^\infty)\to \overline{\BQ}^\times\subset \C^\times$,
$$
\int_{C(p^\infty)}\chi d\mu_j\;=\;
c(\chi,s_{\rm crit})\cdot\frac{L^{(p)}(s_{\rm crit},(\pi\times\sigma)\otimes\chi)}{\Omega_j^{\sgn\chi}},
$$
where
$$
c(\chi,s_{\rm crit})\;=\;G(\chi)^{\frac{(n+1)n}{2}}\cdot\prod_{\mathfrak{p}\mid p}c(\chi_{\mathfrak{p}},s_{\rm crit})
$$
and
$$
c(\chi_{\mathfrak{p}},s_{\rm crit})\;=\;
\begin{cases}
I_{\mathfrak{p}}(\pi,\sigma,t_{\mathfrak{p}},\chi,s_{\rm crit}),
&\text{for unramified $\chi_{\mathfrak{p}}$},\\
\fN(\f_{\chi_{\mathfrak{p}}})^
{\frac{(n+1)n(s_{\rm min}-s_{\rm crit})}{2}-\frac{(n+1)n(n-1)}{6}}
\cdot \kappa_{\mathfrak{p}}^{v_{\mathfrak{p}}(\f_{\chi_\mathfrak{p}})},
&\text{for $\chi_{\mathfrak{p}}$ of conductor $\f_{\chi_{\mathfrak{p}}}\neq 1$}.
\end{cases}
$$
\end{introtheorem}

We remark that our Theorem contains as a particular case a construction of the $p$-adic $L$-functions for $\GL(2)$, valid over any number field. We also remark that there is no restriction on the prime $p$, the oddest prime $2$ is included.

In the Theorem $L^{(p)}(s,\pi\times\sigma)$ denotes the Rankin-Selberg $L$-function with the $p$-Euler factors removed, $s_{\min}=\frac{1}{2}+j_{\min}$ is the left most critical value, and $I_{\mathfrak{p}}(\pi,\sigma,t_{\mathfrak{p}},\chi,s_{\rm crit})$ is a certain local zeta-integral at $\mathfrak{p}\mid p$ whose evaluation is known for $n\leq 2$ by well known computations in the case $n=1$, for $n=2$ by recent yet unpublished work of Ungemach, i.e.\ it agrees with the motivically expected value (cf.\ \cite{coatesperrinriou1989,perrinriou1995}). As the notation suggests, the resulting measure $\mu_j$ depends on choices of Hecke roots at all $\mathfrak{p}\mid p$ for which we find a suitable $U_{\mathfrak{p}}$-eigen vector with eigen value $\kappa_{\mathfrak{p}}$. For details we refer to section \ref{sec:maintheorem}, in particular to our Main Theorem \ref{main:padicl}.

We also prove that the $p$-adic $L$-function we construct statisfies a functional equation. Denote by
$$
\iota:g\mapsto w_{n}g^{-{\rm t}}w_{n}
$$
the twisted main involution of $\GL_{n}$, where the superscript \lq$\,-{\rm t}\,$\rq{} denotes matrix inversion composed with transpose, yielding an outer automorphism of $\GL_{n}$ order $2$. Furthermore $w_n$ is a representative of the long Weyl element for $\GL_n$. Then pullback along $\iota$ sends $\pi$ to its contragredient representation $\pi^\vee$, and a cohomological $U_\mathfrak{p}$-eigen vector $\lambda$ to a cohomological $U_\mathfrak{p}$-eigen vector $\lambda^\vee$. Similarly we have the involution
$$
x\mapsto x^\vee:=(-1)^nx^{-1}
$$
on $C(p^\infty)$. Now the measures $\mu_j$ are constructed out of a specific choice for $\lambda$, and, emphasizing this dependence in the notation, we show in Theorem \ref{thm:functionalequation},
\begin{introtheorem}
We have the functional equation
$$
\mu_{\lambda,j}(x)=
\mu_{\lambda^\vee,-j}(x^{\vee}).
$$
\end{introtheorem}

We point out an intruiging phenomenon which arises whenever $k$ has complex places. Departing from $\pi$ and $\sigma$, or more generally from a $U_p$-eigen class in cohomology, our construction produces a measure (resp.\ distribution) $\tilde{\mu}$ with values in $M^\Gamma$, where $M$ is the irreducible rational $\GL(n+1)\times\GL(n)$-module corresponding to the weight of the regular algebraic representation $\pi\widehat{\otimes}\sigma$, and $\Gamma$ is an arithmetic subgroup of the diagonally embedded $\GL(n)$ in $\GL(n+1)\times\GL(n)$. In order to obtain $\mu_j$, we project $\tilde{\mu}$ onto one-dimensional quotients of $M^\Gamma$. Now if $k$ is totally real, the dimensiona of $M^\Gamma$ coincides with the number of critical places. However if $k$ has a complex place, there are weights for which $\dim M^\Gamma$ is {\em strictly larger} than the number of critical places. It seems unclear so far if these \lq{}phantom components\rq{} of $\tilde{\mu}$, which do not correspond to critical values, hold arithmetic information not covered by the projections $\mu_j$. So far we only know that $\tilde{\mu}$ behaves as nicely as we may expect: It is $p$-adically bounded in the ordinary case, and satisfies a natural functional equation. We hope to return to this question in the future.

We give a quick sketch of the history of the constructions of $p$-adic $L$-functions for $\GL(n+1)\times \GL(n)$. Previously the case of totally real $k$ was settled in \cite{Jan14}, where a less natural condition of ordinarity was formulated, and where the author restricted his attention to a single place. Previously the case of trivial coefficients over an arbitary number field had been treated in \cite{Jan11}. This approach generalized and complemented a construction over $\BQ$ due to Kazhdan, Mazur and Schmidt, who constructed a distribution (but not a measure) in \cite{KMS}, building on previous work of Schmidt \cite{Sch}, who had constructed $p$-adic $L$-functions for regular algebraic representations on $\GL(3)\times\GL(2)$ with trivial central character at infinity. The first general construction of a $p$-adic measure on $\GL(n+1)\times\GL(n)$ over $\BQ$ is Schmidt's complement \cite{Sch2} to \cite{KMS}, which still imposed a trivial central character at infinity and had restrictions on $p$, depending on $n$, in particular excluding small primes. The latter restriction was once and for all removed in \cite{Jan11}, and the correct formulation for arbitrary class numbers for totally real $k$ was given in \cite{Jan14}, which also overcame the restriction of the central character at infinity.

The non-vanishing of the periods is classical in the case $n=1$, which over $\BQ$ corresponds to the case of modular forms, where it is due to Hecke \cite{Hecke, Hecke2, Hecke3}. The case $n=2$ was settled in \cite{Sch} in the case of trivial coefficients over $\BQ$, and for arbitrary coefficients over $\BQ$ in \cite{KS}, which also implies the non-vanishing for totally real fields in this case. The case of general $n$ and general fields with general coefficients was recently proven by Sun \cite{Sun}.

The outline of the paper is as follows. In section \ref{sec:notation} we collect essential notions and notations that will be used continuously throughout the paper. In section \ref{sec:arithmeticmodules} we define our coefficient systems, in section \ref{sec:maintheorem} we state our main theorem in full generality. In sections \ref{sec:distribution} and \ref{sec:padicl} we construct the $p$-adic $L$-function and prove its claimed properties. In section \ref{sec:cohreps} we show that the non-vanishing problem can be treated place by place, which reduces us to two local cases: real and complex places, which allows us to invoke to Sun's result.

By the nature of the problem, there is necessarily some overlap with \cite{Jan14} in our exposition. We use this occasion to simplify the notation and the treatment of the construction of the distribution, and refer to loc.\ cit.\ if technical results are involved which are either already stated in appropriate generality or whose prove generalizes to the situation treated here without change.

The author thanks Binyong Sun for his insistence on simplifications, which significantly contributed to the present form of the article. The author also thanks Miriam Schwab for many corrections.

\section{Notation}\label{sec:notation}

Throughout the paper $k$ denotes a fixed number field and $\OO_k\subseteq k$ its ring of integers. We write $S_\infty$ for the set of archimedean places of $k$ and identify elements of it with conjugacy classes of field embeddings $\iota:k\to\BC$, via the action of complex conjugation on the codomain. In this spirit we denote by $\overline{\iota}$ the complex conjugate of an embedding $\iota$. By abuse of notation we sometimes identify $\iota$ with its class $\{\iota,\overline{\iota}\}$. We decompose $S_\infty=S_{\rm real}\cup S_{\rm cplx}$ disjointly into the sets of real resp.\ complex places.

We write $\BA_k$ for the adele group over $k$, and for a finite set $S$ of places of $k$ we denote by $\BA_k^{(S)}$ the adeles where the finitely many components corresponding to places in $S$ have been removed (or are identically $1$ or $0$, depending on the context). We apply same notation to ideles and write $\infty$ for the set of infinite places of $k$. Thus $\BA_k^{(\infty)}$ denotes the ring of {\em finite} adeles.

We fix an algebraic closure $\overline{\BQ}/\BQ$ and assume that $k\subseteq\overline{\BQ}$. We fix embeddings $\iota_\infty:\overline{\BQ}\to\BC$ and $\iota_p:\overline{\BQ}\to\overline{\BQ}_p$, the latter denoting an algebraic closure of $\BQ_p$, and $p$ is a fixed rational prime throughout the paper. We emphasize that we do {\em not} exclude any small primes.

If $l/\BQ$ is an extension with ring of integers $\OO_l$, we let $\OO_{l,(p)}=\OO_l\otimes_\BZ \BZ_{(p)}$ denote the localization of $\OO_l$ at $p$, i.e.\ the subring of $p$-integral elements in $l$.

 For a linear algebraic group $G$ over a field $l$, we write $X(G)$ for the group of characters of $G$, i.e.\ homorphisms $G\to\Gm$, defined over a separable algebraic closure of $l$. If $m/l$ is an extension, let $X_m(G)$ denote the characters defined over $m$.

We write $G_n := \res_{\OO_k/\BZ}\GL_{n}$ for the restriction of scalars of the group scheme $\GL_{n}/\OO_k$. We fix the standard maximal torus $T_n$ in $\GL_n$ and consider all root data for $\GL_n$ resp.\ $G_n$ with respect to $T_n$ resp.\ $\res_{k/\BQ}T_n$ and the ordering induced by the standard Borel subgroup $B_n=T_nN_n$ resp.\ $\res_{k/\BQ}B_n$. We make the usual standard choice of simple roots, write $W(\GL_n,T_n)$ for the absolute Weyl group, $w_n\in W(\GL_n,T_n)$ for the longest element with respect to our choice of positive system, and we choose as basis of characters the component projections $\chi_j:T_n\to\Gm$, $(t_i)\mapsto t_j$. We write $X(H)$ for the character group of a group scheme $H$ and use the basis $\chi_1,\dots,\chi_n$ to fix an isomorphism $X(T_n)\cong\BZ^n$ once and for all.

Every weight $\boldsymbol{\mu}\in X(\res_{k/\BQ}T_n)$ may be naturally identified with a tuple
$$
\boldsymbol{\mu}\;=\;(\mu_\iota)_{\iota\in\Hom(k,\BC)}
$$
where $\mu_\iota\in X(T_n)$. Identifying a place $v\in S_\infty$ with a set of embeddings we set
$
\mu_v:=(\mu_\iota)_{\iota\in v}
$
By the very definition of the restriction of scalars the Galois action on $X(\res_{k/\BQ}T_n)$ is given by
$$
\boldsymbol{\mu}^\sigma\;=\;(\mu_{\sigma^{-1}\iota})_{\iota\in \Hom(k,\BC)}
$$
for $\sigma\in\Aut(\BC/\BQ)$, and the irreducible representation $(\rho_{\boldsymbol{\mu}},M_{\boldsymbol{\mu}})$ of highest weight $\boldsymbol{\mu}$ of $G_n$ is defined over the field of rationality $\BQ(\boldsymbol{\mu})$ of $\boldsymbol{\mu}$. We denote the latter's ring of integers with $\OO(\boldsymbol{\mu})$ and fix a $G_n(\OO(\boldsymbol{\mu}))$-stable lattice $M_{\mu}(\OO({\boldsymbol{\mu}}))$ in the representation space of $M_{\boldsymbol{\mu}}(\BQ(\boldsymbol{\mu}))$ and thus regard the latter as an $\OO(\boldsymbol{\mu})$-scheme of representations of $G_n\times_{\BZ}\OO(\boldsymbol{\mu})$. I.e.\ for any $\OO(\boldsymbol{\mu})$-algebra we have the set of $A$-valued points
$$
M_{\boldsymbol{\mu}}(A)\;=\;
M_{\boldsymbol{\mu}}(\OO(\boldsymbol{\mu}))\otimes_{\OO(\boldsymbol{\mu})}A,
%A^{\rank(M_{\boldsymbol{\mu}})},
$$
which is a representation of $G_n(A)$. We let
$$
M_{\boldsymbol{\mu},(p)}\;:=\;
M_{\boldsymbol{\mu}}\times_{\OO(\boldsymbol{\mu})}\OO(\boldsymbol{\mu})_{(p)},
$$
denote the localization at $p$. We remark that we have for any $\OO(\boldsymbol{\mu})_{(p)}$-algebra a natural isomorphism
$$
M_{\boldsymbol{\mu}}(A)\;=\;
M_{\boldsymbol{\mu},(p)}(A)\;=\;
A^{\rank(M_{\boldsymbol{\mu}})}.
$$

We write $M_{\boldsymbol{\mu}}^\vee$ for the (algebraic) dual, a notation that we extend to arbitrary representations. We fix once and for all an identification
$$
M_{\boldsymbol{\mu},(p)}^\vee\;=\;M_{\boldsymbol{\mu}^\vee,(p)},
$$
where $\boldsymbol{\mu}^\vee:=-w_n\boldsymbol{\mu}$.

We set
$$
{}^0G_n:=\bigcap_{\alpha\in X_\BQ(G_n)}\Ker(\alpha^2).
$$
Then ${}^0G_n$ is a reductive group scheme over $\BZ$ and $X_\BQ({}^0G_n)\otimes_\BZ\BQ=0$ \cite[Proposition 1.2]{Jan11}.

Furthermore denote by $S$ the maximal $\BQ$-split torus in the radical of $G_n$, or the maximal $\BQ$-split central torus of $G_n$, what amounts to the same. Then
\begin{equation}
G_n(\BR)={}^0G_n(\BR)\rtimes S(\BR)^0,
\label{eq:gn0gns}
\end{equation}
cf. \cite[Proposition 1.2]{borelserre1973}. We have explicitly
$$
{}^0G_n(\BR)=\{g\in G_n(\BR)\mid \prod_{\iota\in S_\infty}\norm{\det g_\iota}_\iota=1\}.
$$
This in particular shows that $S$ is a $\BQ$-split torus of rank $1$. Furthermore we introduce the subgroup
$$
G_n^\pm=\{g\in G_n(\BR)\mid \forall \iota\in S_\infty:\norm{\det g_\iota}_\iota=1\}.
$$

We fix standard maximal compact subgroups
$$
\begin{CD}
K_n@>\subseteq>>G_n(\BR)\\
@| @|\\
\prod_{\iota\in S_\infty}K_\iota@>\subseteq>>
\prod_{\iota\in S_\infty}\GL_n(k_{\iota}),
\end{CD}
$$
where $k_{\iota}$ denotes the completion of $k$ at $\iota$ and
$$
K_\iota\;=\;
\begin{cases}
\oU(n),&\text{for $\iota$ complex},\\
\oO(n),&\text{for $\iota$ real}.
\end{cases}
$$
Then by \eqref{eq:gn0gns}, $K_n$ eventually lies in ${}^0G_n(\BR)$, as does every arithmetic subgroup of $G_n$.

We write $Z_n$ for the center of $G_n$ and set
$$
GK_n\;:=\;K_n Z_n(\BR)^0,
$$
then, by the above observation, $GK_n$ is again a product of groups $\GO(n)$ and $\GU(n)$ at real and complex places respectively, which are defined as
\[
  \GO(n):=\{g\in \GL_n(\R)\mid g^{\mathrm t} g\in\R^\times{\bf1}_n\},
  \]
and
\[
  \GU(n):=\{g\in \GL_n(\BC)\mid \bar{g}^{\mathrm t} g\in\R^\times{\bf1}_n\}.
  \]
Here and henceforth, a superscript \lq$\,\mathrm t\,$\rq{} over a matrix indicates its transpose, a superscript \lq$\,0\,$\rq over a Lie or an algebraic group indicates its identity connected component, $\pi_0(G):=G/G^0$ denotes the resulting component group.

%We write $\theta_n$ for the Cartan involutions on $\GL_n(\BR)$, $\GL_n(\BC)$, and $G_n(\BR)$ corresponding to the above choice of standard maximal compact subgroups. We use the same notation for the Cartan involusions induced on the complexified Lie algebras $\gl_n$, $\gl_n^\BC$, $\g_n$.

We use German gothic letters to indicate the complexified Lie algebras of the corresponding groups.

The superscripts \lq$\,\ad\,$\rq{} and \lq$\,\der\,$\rq{} over a group or a Lie algebra indicate the corresponding adjoint or derived group or algebra respectively.

For local archimedean considerations we consider $\GL_n$ as a reductive group over $\BQ$ and set
\[
  \SL_n^{\pm}:={}^0\GL_n.
  \]
Then
\[
  \SL_n^{\pm}(\R)=\{g\in \GL_n(\R)\mid \det(g)=\pm 1\}.
  \]

Recall that a representation of a real reductive group is called a Casselman-Wallach representation if it is smooth, Fr\'{e}chet, of moderate growth, and its Harish-Chandra module has finite length.

As usual, a superscript group indicates the space of invariants of a group or Lie algebra representation.

For any quasi-simple $G_n(\BR)$-representation $V$ we write $V^{(K_n)}$ for the subspace of (smooth) $K_n$-finite vectors. Then we have
\begin{equation}
H^\bullet(\g_n,GK_n; V^{(K_n)})\;=\;
H^\bullet(\g_n^{\rm der},K_n^{\rm der}; H^0(Z_n(\BR)^0;
V^{(K)})).
\label{eq:ggkcohomology}
\end{equation}
In particular $(\g_n,GK_n)$-cohomology behaves the same way as classical $(\g_n^{\rm der},K_n^{\rm der})$-cohomology.

We set
$$
b_{n}^{\BR}\;:=\;\left\lfloor\frac{n^2}{4}\right\rfloor,
$$
$$
b_{n}^{\BC}\;:=\;\frac{n(n-1)}{2},
$$
and
$$
b_{n}^k\;:=\;\sum_{v\in S_\infty} b_{n}^{k_v}.
$$
This will turn out to be the bottom degree for the cohomology \eqref{eq:ggkcohomology} for the infinity-component of the regular algebraic representations we consider.

We fix the standard embedding
$$
j:\;\;\;\GL_{n}\;\to\;\GL_{n+1},
$$
$$
\;\;\;\;\;\;\;\;\;\;\;\;\;\;\;
g\;\mapsto\;
\begin{pmatrix}
g&0\\
0&1
\end{pmatrix},
$$
which induces an embedding $j:G_{n}\to G_{n+1}$, and all of the above data are compatible with it. The corresponding diagonal embedding is denoted
\begin{equation}
j\times 1:\;\;\;G_{n}\to G_{n+1}\times G_n.
\label{eq:diagonalembedding}
\end{equation}

For every finite place $\mathfrak{p}\mid p$ of $k$, let $k_{\mathfrak{p}}$ denote the completion of $k$ at $\mathfrak{p}$, $\OO_\mathfrak{p}\subseteq k_\mathfrak{p}$ its valuation ring, and we normalize the norm such that
$$
|a|_\mathfrak{p}=|N_{k_\mathfrak{p}/\BQ_p}(a)|_p=\absNorm(\OO_\mathfrak{p}a)^{-1},
$$
the latter denoting the absolute norm. All Haar measures are normalized in such a way that a standard maximal compact subgroup (if it exists) receives Haar measure $1$.

We fix once and for all a character $\psi:k\backslash\BA_k\to\BC^\times$ which has a factorization $\psi=\otimes_v \psi_v$. Then $\psi$ induces a character $U_n(\BA_k)\to\BC^\times$, for $u=(u_{ij})$ explicitly given as
$$
\psi(u)=\sum_{i=1}^{n-1}u_{ii+1},
$$
the same applies locally. For convenience of normalization in Theorem \ref{thm:localbirch} we assume that $\psi_{\mathfrak{p}}$ has conductor $\OO_{\mathfrak{p}}$ for $\mathfrak{p}\mid p$. If $V$ is a generic representation of $\GL_n(\BA_k)$ we write ${\mathscr{W}}(V,\psi)$ for the $\psi$-Whittaker space associated to $V$. Again the same notation applies locally. Our choice of $\psi$ also fixes the local and global Gau\ss{} sums as
$$
G(\chi)\;:=\;\sum_{a\mod\mathfrak{f}_\chi}
\chi\left(\frac{a}{f}\right)
\psi\left(\frac{a}{f}\right),
$$
where $f$ is a generator of the conductor $\mathfrak{f}_\chi$ of the given quasi-character $\chi:k^\times\backslash\BA_k^\times\to\BC^\times$.

We write $K_{\mathfrak{p}}(m)$ for the mod-$\mathfrak{p}^m$-Iwahori subgroup of $\GL_{n+1}(k_{\mathfrak{p}})$, and similarly $K_{\mathfrak{p}}'(m)$ for the respective Iwahori subgroup of $\GL_n(k_{\mathfrak{p}})$.

Choose a uniformizer $\varpi_{\mathfrak{p}}$ in the valuation ring $\OO_{\mathfrak{p}}\subseteq k_{\mathfrak{p}}$ and set
$$
t_{(\mathfrak{p})}\;:=\;
\begin{pmatrix}
\varpi_{\mathfrak{p}}^n&&&\\
&\varpi_{\mathfrak{p}}^{n-1}&&\\
&&\ddots&\\
&&&1
\end{pmatrix}.
$$
For any ideal $\mathfrak{f}\mid p^\infty$ we fix a generator
$$
f\;:=\;\prod_{\mathfrak{p}\mid p}\varpi_{\mathfrak{p}}^{v_\mathfrak{p}(\mathfrak{f})},
$$
and set
$$
t_f\;:=\;
\prod_{\mathfrak{p}\mid p}
t_{(\varpi_{\mathfrak{p}})}^{v_\mathfrak{p}(\mathfrak{f})}.
$$
We introduce the magic matrix
$$
h^{(1)}:=
\begin{pmatrix}
&&&1\\
&w_{n}&&\vdots\\
&&&\vdots\\
0&\hdots&0&1
\end{pmatrix}\in\GL_{n+1}(\BZ),
$$
and define
$$
h^{(f)}\;:=\;t_{(f)}^{-1}
\cdot h^{(1)}\cdot
t_{(f)}
\;\in\;G_{n+1}(\BQ_p),
$$
We consider the Hecke operator
$$
V_{\mathfrak{p}}:=
K_{\mathfrak{p}}(m) t_{(\mathfrak{p})}
K_{\mathfrak{p}}(m),
$$
for $G_{n+1}$ and
$$
V_{\mathfrak{p}}'=
K_{\mathfrak{p}}'(m) \varpi_{\mathfrak{p}}t_{(\mathfrak{p})}
K_{\mathfrak{p}}'(m)
$$
as a Hecke operator for $G_n$. Then we have the Hecke operator
$$
U_{\mathfrak{p}}\;:=\;V_\mathfrak{p}\otimes V_\mathfrak{p}'
$$
acting on $(\pi\otimes\sigma)^{K_{\mathfrak{p}}(m)\times K_{\mathfrak{p}}'(m)}$. Those Hecke operators may be considered as products of the standard Hecke operators
$$
T_{\mathfrak{p},\nu}\;:=\;K_{\mathfrak{p}}(m)\begin{pmatrix}\varpi\cdot{\bf1}_{\nu}&0\\0&{\bf1}_{n+1-\nu}\end{pmatrix}K_{\mathfrak{p}}(m),\;\;\;0\leq\nu\leq n+1,
$$
and
$$
T_{\mathfrak{p},\nu}'\;:=\;K_{\mathfrak{p}}'(m)\begin{pmatrix}\varpi\cdot{\bf1}_{\nu}&0\\0&{\bf1}_{n-\nu}\end{pmatrix}K_{\mathfrak{p}}'(m),\;\;\;0\leq\nu\leq n.
$$

\section{Arithmetic modules}\label{sec:arithmeticmodules}

We call a simple $G_n$-module $M$, defined over a number field $\BQ(M)/\BQ$, {\em prearithmetic} if it is {\em essentially conjugate self-dual over $\BQ$}, i.e.\
\begin{equation}
M^{\vee,c}\;\cong\;M\otimes\xi
\label{eq:mduality}
\end{equation}
where $\xi\in X_{\BQ}(G_n)$, i.e.\ $\xi$ is defined over $\BQ$. We remark that the restriction map
$$
X_{\BQ}(G_n)
\;\to\;
X(S)
$$
is a monomorphism with finite cokernel. This map factors over $X_{\BQ}(\res_{k/\BQ}T_n)$, and its image in the latter consists of constant tuples $(w)_{\iota}$, $w\in\BZ=X(S)$. As there is no risk of confusion we simply write $(w)\in X_{\BQ}(\res_{k/\BQ}T_n)$.

Then an irreducible $G_n$-module $M_{\boldsymbol{\mu}}$ of highest weight $\mu$ is {\em prearithmetic} if and only if
\begin{equation}
\boldsymbol{\mu}-w_n\boldsymbol{\mu}^c\;=\;w\;\in\;X(S).
\label{eq:mupurity}
\end{equation}
We call $M_{\boldsymbol{\mu}}$ and also $\boldsymbol{\mu}$ {\em arithmetic} if $M_{\boldsymbol{\mu}^\sigma}$ is prearithmetic for every $\sigma\in\Aut(\BC/\BQ)$.

The motivation for our terminology is that the existence of an automorphic representation $\pi$ of $G_n(\BA)$ with non-trivial cohomology with coefficients in $M_{\boldsymbol{\mu}}(\BC)$ implies the arithmeticity of $\boldsymbol{\mu}$ (cf.\ \cite[Th\'eor\`eme 3.13 and Lemme 4.9]{Clo}).

We have the following elementary
\begin{prp}
Let $\boldsymbol{\mu}$ be arithmetic. Then
\begin{itemize}
\item[(i)] its \lq{}Tate twists\rq{} $\boldsymbol{\mu}+(j)$ for any $j\in\BZ$ are arithmetic,
\item[(ii)] its dual weight $-w_n\boldsymbol{\mu}$ is arithmetic,
\item[(iii)] the conjugate weights $\boldsymbol{\mu}^\sigma$ for $\sigma\in\Aut(\BC/\BQ)$ are arithmetic.
\end{itemize}
\end{prp}

%We call a weight $\overline{\mu}=(\mu_\iota)_{\iota}$ of $G_n$ {\em pure} if it satisfies the following condition:
%\begin{itemize}
%\item[(P)] The quantity $\mu_{\iota,i}+\mu_{\overline{\iota},n+1-i}$ is independent of $\iota\in S_\infty$, $1\leq i\leq n$.
%\end{itemize}
%In that case we call
%$$
%w(\boldsymbol{\mu})\;:=\;\mu_{\iota,i}+\mu_{\overline{\iota},n+1-i}
%$$
%the {\em motivic weight} of $\boldsymbol{\mu}$.

Now let $\boldsymbol{\mu}$ resp.\ $\boldsymbol{\nu}$ be arithmetic weights of $\res_{k/\BQ}T_{n+1}$ resp.\ $\res_{k/\BQ}T_n$, $n\geq 1$. We call a $j\in\BZ$ {\em critical} for $\boldsymbol{\mu}\times\boldsymbol{\nu}$ if there is a non-zero $G_n$-equivariant map
\begin{equation}
\xi_j:\;\;\;M_{\boldsymbol{\mu}}\otimes M_{\boldsymbol{\nu}}\;\to\;M_{(j)}.
\label{eq:tauij}
\end{equation}
Assuming its existence, $\xi_j$ is defined over $\OO(\boldsymbol{\mu},\boldsymbol{\nu})$ (the ring of integers in $\BQ(\boldsymbol{\mu},\boldsymbol{\nu})$), and is unique up to an integral unit in the latter ring due to the multiplicity one property of restrictions of irreducibles for $\GL_{n+1}|\GL_n$.

We write
$$
\xi\;:=\;\oplus_{j}\xi_j:M_{\boldsymbol{\mu}}\otimes M_{\boldsymbol{\nu}}\to\bigoplus_{j}M_{(j)}=:M_{(\boldsymbol{\mu},\boldsymbol{\nu})},
$$
where $j$ runs through the critical $j$ for $\boldsymbol{\mu}\times\boldsymbol{\nu}$. Let $j_{\min}$ denote the minimal critical $j$.

We write $\Gamma_n$ for the Zariski-closure, i.e.\ the {\em algebraic hull}, of a torsion-free arithmetic subgroup $\Gamma$ of $G_n$ inside $G_n$. Note that $\Gamma_n$ is independent of $\Gamma$.

\begin{lem}\label{lem:gammannormal}
We have $G_{n}^{\der}\subseteq\Gamma_n$ and the latter is normal in $G_{n}$.
\end{lem}

\begin{proof}
We know that
$$
\Gamma_n\cap G_{n}^{\rm der}=G_{n}^{\rm der},
$$
because $\Gamma\cap G_{n}^{\rm der}$ is an arithmetic subgroups of $G_n^{\rm der}$ and hence Zariski-dense in $G_{n}^{\rm der}$. This shows
$$
G_{n}^{\rm der}\;\subseteq\; \Gamma_n.
$$
As the quotient
$$
G_{n}/G_{n}^{\rm der}
$$
is a torus, it is abelian and hence $\Gamma_n$ is normal in $G_{n}$.
\end{proof}

Now let $\pi$ resp.\ $\sigma$ be regular algebraic irreducible cuspidal automorphic representations of $G_{n+1}(\BA)$ resp.\ $G_{n}(\BA)$ with weights $\boldsymbol{\mu}$ resp.\ $\boldsymbol{\nu}$, i.e.\
\begin{equation}
H^\bullet(\g_{n+1},GK_{n+1};
\pi_\infty^{(K_{n+1})}\otimes_\BC M_{\boldsymbol{\mu}}(\BC))\;\neq\;0,
\label{eq:pigkcoh}
\end{equation}
and
$$
H^\bullet(\g_{n},GK_{n};
\sigma_\infty^{(K_{n})}\otimes_\BC M_{\boldsymbol{\nu}}(\BC))\;\neq\;0.
$$
The following observation is due to Kasten-Schmidt \cite{KS} (at real places) and Raghuram \cite{Ragpre} (complex places).
\begin{prp}
The map
$$
j\;\mapsto\;\frac{1}{2}+j
$$
sets up a bijection between integers $j\in\BZ$ critical for $\boldsymbol{\mu}\times\boldsymbol{\nu}$ and critical values of $L(s,\pi\times\sigma)$ in the sense of Deligne \cite{Del}.
\end{prp}

\section{The Main Theorem}\label{sec:maintheorem}

Let $\pi$ and $\sigma$ denote irreducible cuspidal regular algebraic automorphic representations as in the previous section, that we also consider as representations of $\GL_{n+1}(\BA_k)$ and $\GL_{n}(\BA_k)$. Then we know by Clozel \cite[Th\'eor\`eme 3.13 resp.\ Proposition 3.16]{Clo}, that the finite components $\pi^{(\infty)}$ and $\sigma^{(\infty)}$ are defined over a number field $E=\BQ(\pi,\sigma)\supseteq\BQ(\boldsymbol{\mu},\boldsymbol{\nu})$, which is the field of rationality of $\pi^{(\infty)}\otimes\sigma^{(\infty)}$. We remark that in an appropriate sense, $\pi\widehat{\otimes}\sigma$ as a global representation is defined over the field $\BQ(\pi,\sigma)$ as well \cite{Janpre}.

We fix a prime $p$ subject to the following conditions:
\begin{itemize}
%\item[(I)] for each prime $\mathfrak{p}\mid p$ in $k$, the representation spaces $\pi_{\mathfrak{p}}$ and $\sigma_{\mathfrak{p}}$ possess non-zero Iwahori-invariant vectors of level $\mathfrak{p}^m$, $m\geq 1$.
\item[(H)] for each prime $\mathfrak{p}\mid p$ in $k$, there is a $U_{\mathfrak{p}}$-eigen vector
$$
0\neq t_{\mathfrak{p}}\in(\pi_{\mathfrak{p}}\otimes_\BC\sigma_{\mathfrak{p}})^{K_{\mathfrak{p}}(m)\times K_{\mathfrak{p}}'(m)},
$$
and the corresponding eigen values $\kappa_{\mathfrak{p}}\in E$ are all {\em non-zero}.
\end{itemize}
We call the latter the {\em finite slope} condition. We also consider the stronger $\mathfrak{p}$-{\em ordinarity} condition
\begin{itemize}
\item[(O)] If $s_{\min}\in\frac{1}{2}+\BZ$ denotes the left most critical value for $L(s,\pi\times\sigma)$, then for all $\mathfrak{p}\mid p$
$$
|\kappa_{\mathfrak{p}}|_p=\left|
\boldsymbol{\mu}(a_{\mathfrak{p}})\cdot
\boldsymbol{\nu}(a_{\mathfrak{p}}')
%\absNorm(\mathfrak{p})^{(s_{\min}-\frac{1}{2})\cdot\frac{n(n+1)}{2}}
\right|_p,
$$
\end{itemize}
where
$$
a_{(\mathfrak{p})}'\;:=\;
\begin{pmatrix}
\varpi_{\mathfrak{p}}^{-n}&&&\\
&\varpi_{\mathfrak{p}}^{-n+1}&&\\
&&\ddots&\\
&&&\varpi_{\mathfrak{p}}^{-1}
\end{pmatrix},
$$
in $\GL_n(k)$ and $a_{({\mathfrak{p}})}=j(a_{({\mathfrak{p}})}')$, with local uniformizers $\varpi_{\mathfrak{p}}\in\OO_{\mathfrak{p}}$.

We emphasize that these are conditions on the tuple $(\pi,\sigma,(t_{\mathfrak{p}})_{\mathfrak{p}\mid p})$.

For $\pi_\mathfrak{p}$ resp.\ $\sigma_\mathfrak{p}$ spherical, with pairwise distinct Hecke roots, it is well known that the corresponding data are all of finite slope (cf.\ \cite[Proposition 4.12]{KMS} resp.\ \cite[Section 1.6]{Jan14}).

We write $L(s,\pi\times\sigma)$ for the Rankin-Selberg $L$-function attached to $(\pi,\sigma)$ in the sense of \cite{jpss1983} and assume in the sequel that there is at least one $s\in\frac{1}{2}+\BZ$ which is critical for $L(s,\pi\times\sigma)$ in the sense of \cite{Del} and we let $s_{\min}=\frac{1}{2}+j_{\min}$ denote the left most critical value.

As our representations at $\mathfrak{p}$ are generic we may consider the Fourier isomorphism
$$
w:\;\;\;\pi_{\mathfrak{p}}\otimes\sigma_{\mathfrak{p}}\to
{\mathscr{W}}(\pi_{\mathfrak{p}},\psi_{\mathfrak{p}})
\otimes
{\mathscr{W}}(\sigma_{\mathfrak{p}},\psi_{\mathfrak{p}}^{-1}).
$$

We let for any place $v$ of $k$
$$
L_v\;:=\;
\begin{cases}
\Gm(\OO_v),&v\nmid\infty,\\
\Gm(k_v)^0,&v\mid\infty.
\end{cases}
$$
Then for any ideal $\mathfrak{f}\mid p^\infty$ in $k$, i.e.\ any ideal dividing a $p$-power, we define
$$
C(\mathfrak{f})\;:=\;
k^\times\backslash\BA_k^\times/(1+\mathfrak{f})\prod_{v\nmid \mathfrak{f}}L_v,
$$
and set
$$
C(p^\infty)\;:=\;
\varprojlim C(\mathfrak{f})=
\BA_k^\times/\overline{k^\times\prod_{v\nmid p}L_v},
$$
where $\mathfrak{f}$ ranges over the ideals dividing $p^\infty$. By class field theory $C(p^\infty)$ corresponds isomorphically to the Galois group of the maximal abelian extension $k^{(p^\infty)}$ of $k$ unramified outside $p\infty$.

\begin{thm}\label{main:padicl}
Under the above condition (O) resp.\ (H), for each $s_{\rm crit}=\frac{1}{2}+j$ critical for $L(s,\pi\times\sigma)$, there exist periods $\Omega_j^{\pm}\in\BC^\times$ and an $\BQ(\pi,\sigma)$-valued $p$-adic measure $\mu_j$ resp.\ a distribution in the finite slope case on $C(p^\infty)$, such that for all finite order characters $\chi:C(p^\infty)\to \overline{\BQ}$:
$$
c(\chi,s_{\rm crit})\cdot
\frac{L^{(p)}(s_{\rm crit},(\pi\times\sigma)\otimes\chi)}{\Omega_j^{\sgn\chi}}\;=\;
\int_{C(p^\infty)}\chi d\mu_j,
$$
with
$$
c(\chi,s_{\rm crit})\;=\;G(\chi)\cdot\prod_{\mathfrak{p}\mid p}c(\chi_{\mathfrak{p}},s_{\rm crit})
$$
and
$$
c(\chi_{\mathfrak{p}},s_{\rm crit})\;=\;
\begin{cases}
I_{\mathfrak{p}}(\pi,\sigma,t_{\mathfrak{p}},\chi,s_{\rm crit}),
%I_{\mathfrak{p}}(\pi,\sigma,t_{\mathfrak{p}},\underline{\lambda},\underline{\lambda}',\chi),
&\text{for unramified $\chi_{\mathfrak{p}}$},\\
\fN(\f_{\chi_{\mathfrak{p}}})^{\frac{(n+1)n(s_{\rm min}-s_{\rm crit})}{2}-\frac{(n+1)n(n-1)}{6}}
\cdot \kappa_{\mathfrak{p}}^{v_\mathfrak{p}(\f_{\chi_\mathfrak{p}})},
&\text{for $\chi_{\mathfrak{p}}$ of conductor $\f_{\chi_{\mathfrak{p}}}$}.
\end{cases}
$$
Furthermore
$$
I_{\mathfrak{p}}(\pi,\sigma,t_{\mathfrak{p}},\chi,s_{\rm crit})
:=
\int_{U_{n}(k_{\mathfrak{p}})\backslash\GL_{n}(k_\mathfrak{p})}
w(t_{\mathfrak{p}})
\left(
j(g)
%\begin{pmatrix}
%g&\\&1
%\end{pmatrix}
h^{(\varpi)}
,g
\right)
\chi_{\mathfrak{p}}(\det(g))|\det(g)|_{\mathfrak{p}}^{\nu}dg.
$$
\end{thm}

Up to the computation of the Euler factors at the finite places $v\nmid\mathfrak{p}$ where $\det(K)\neq\OO_v^\times$, Theorem \ref{main:padicl} generalizes by construction to arbitrary finite order characters of $C(K(p^\infty))$. As already remarked in the introduction, the value of this zeta integral is known for normalized choices of $t_{\mathfrak{p}}$ in the cases $n\leq 2$, assuming $\pi$ and $\sigma$ unramified at $\mathfrak{p}$, and in those cases agrees with the value expected in \cite{coatesperrinriou1989,perrinriou1995}.

The construction of the distribution and its main properties will be proven in Theorems \ref{thm:distribution}, \ref{thm:boundedness}, \ref{thm:interpolation}. We also remark that the $p$-adic $L$-function we construct satisfies the equation as stated in the introduction, see Theorem \ref{thm:functionalequation}.

\section{Topological formalism}\label{sec:sheaves}

\subsection{Locally symmetric spaces}

For each compact open subgroup $K$ of $G_{n+1}(\BA_\BQ^{(\infty)})$ we consider the locally symmetric spaces
$$
\mathscr X_{n+1}(K)\;:=\;
G_{n+1}(\BQ)\backslash G_{n+1}(\BA_\BQ)/K_{n+1}K,
$$
and
$$
\mathscr X_{n+1}^{\ad}(K)\;:=\;
G_{n+1}(\BQ)\backslash G_{n+1}(\BA_\BQ)/GK_{n+1}K,
$$
which come with the canonical projection map denoted
$$
\ad:\quad
\mathscr X_{n+1}(K)\to\mathscr X_{n+1}^{\ad}(K).
$$
We set
$$
C(K)\;:=\;
k^\times\backslash\BA_k^\times/(k\otimes_\BQ\BR)^0\det(K),
$$
which is a finite discrete topological space, and comes with the map
$$
{\det}_K:\;\;\;\mathscr X_n(K)\to k^\times\backslash\BA_k^\times/(k\otimes_\BQ\BR)^0\det(K),
$$
factoring over $\ad$. In particular we write
$$
\mathscr X_{n+1}(K)[c]
\;:=\;
{\det}_K^{-1}(c)
$$
for the fiber of a class
$$
c\in C(K),
$$
for which we may choose a representative $c\in\BA_k^\times$ which is trivial at infinity. Then by strong approximation we have a natural identificaton
\begin{equation}\label{eq:symmetriccomponent}
\mathscr X_{n+1}(K)[c]\;=\;\Gamma_c\backslash G_{n+1}(\BR)/K_{n+1},
\end{equation}
where
$$
\Gamma_c:=G_{n+1}(\BQ)\cap d_{(c)}Kd_{(c)}^{-1},
$$
with
$$
d_{(c)}:=\diag(c,1,\dots,1)\;\in\;G_n(\BA_\BQ).
$$
In particular all fibers are connected. In the squel we assume $K$ to be {\em neat} in the sense that for all $c\in C(K)$ the arithmetic group $\Gamma_c$ is torsion free (although $\Gamma_c$ as a subgroup of $G_n(\BQ)$ depends on the chosen representative of $c$, its isomorphism class as a group doesn't), which can always be achieved by replacing $K$ with a suitable neat compact open subgroup (which then necessarily is of finite index in $K$). For neat $K$ the space $\mathscr X_{n+1}(K)$ is a manifold.

For a sheaf $F$ on $\mathscr X_{n+1}(K)$ we denote by $F[c]$ its restriction to $\mathscr X_{n+1}(K)[c]$.

Fix a field extension $E/\BQ(\boldsymbol{\mu})$. Then the $G_{n+1}(\BQ)$-module $M_{\boldsymbol{\mu}}(E)$ gives rise to a sheaf $\underline{M}_{\boldsymbol{\mu}}(E)$ on $\mathscr X_{n+1}(K)$, arising as a local system from the left- $G_{n+1}(\BQ)$- and right- $K_n^0\times K$-module
$$
\widetilde {M}_{\boldsymbol{\mu}}(E)\;:=\;
{M}_{\boldsymbol{\mu}}(E)\times G_{n+1}(\A_\BQ).
$$
Its sections on an open subset $U\subseteq\mathscr X_{n+1}(K)$ are explicitly given by
$$
\Gamma(U;\underline{M}_{\boldsymbol{\mu}}(E))=
\{f:G_{n+1}(\BQ)U\to M_{\boldsymbol{\mu}}(E)\mid f\;\text{locally constant and}
$$
$$
\forall\gamma_{n+1}\in G_{n+1}(\BQ),u\in G_{n+1}(\BQ)U:f(\gamma_{n+1} u)=\rho_{\boldsymbol{\mu}}(\gamma_{n+1})f(u)
\}.
$$
We introduce the translations
$$
t_g:\;\;\;\mathscr X_{n+1}(gKg^{-1})[c]\to \mathscr X_{n+1}(K)[c\det(g)],
$$
$$
G_{n+1}(\BQ)x gKg^{-1}\mapsto G_{n+1}(\BQ)xg K,
$$
for $g\in G_{n+1}(\BA_\BQ^{(\infty)})$. Then we have the identification
\begin{equation}
\Gamma(U;t_g^*\underline{M}_{\boldsymbol{\mu}}(E))=
\Gamma(t_g(U);\underline{M}_{\boldsymbol{\mu}}(E)).
\label{eq:pullbackid}
\end{equation}
The pullbacks $t_g^*$ naturally extend to cohomology.

In order to lay hands on the pullbacks, we consider for each open $U\subseteq\mathscr X_{n+1}(gKg^{-1})$ the map
$$
\Gamma(U;t_g^*\underline{M}_{\boldsymbol{\mu}}(E))
\;\to\;\Gamma(U;\underline{M}_{\boldsymbol{\mu}}(E)),
$$
given via the identification \eqref{eq:pullbackid} by
$$
\Gamma(t_g(U);\underline{M}_{\boldsymbol{\mu}}(E))\;\ni\;
f\;\mapsto\;\left[u\mapsto f(ug)\right],
$$
i.e.\ we obtain a map
$$
\rho_{\boldsymbol{\mu}}(g):\quad
t_g^*\underline{M_{\boldsymbol{\mu}}}(E)\to \underline{M}_{\boldsymbol{\mu}}(E)
$$
of sheaves on $\mathscr X_{n+1}(gKg^{-1})$.

In particular the composition
$$
t_g^{\boldsymbol{\mu}}\;:=\;\rho_{\boldsymbol{\mu}}(g)\circ t_g^*
$$
sends sections of the sheaf $\underline{M}_{\boldsymbol{\mu}}(E)$ on $\mathscr X_{n+1}(K)$ to sections of the sheaf $\underline{M}_{\boldsymbol{\mu}}(E)$ on $\mathscr X_{n+1}(gKg^{-1})$. By functoriality this gives a map
$$
t_g^{\boldsymbol{\mu}}:\quad
H_*^q(\mathscr X_{n+1}(K),\underline{M}_{\boldsymbol{\mu}}(E))\to
H_*^q(\mathscr X_{n+1}(gKg^{-1}),\underline{M}_{\boldsymbol{\mu}}(E)),
$$
which has the factorization
$$
\begin{CD}
H_*^q(\mathscr X_{n+1}(K),\underline{M}_{\boldsymbol{\mu}}(E))
@>t_g^*>>
H_*^q(\mathscr X_{n+1}(gKg^{-1}),t_g^*\underline{M}_{\boldsymbol{\mu}}(E))\\
@>\rho_{\boldsymbol{\mu}}(g)>>
H_*^q(\mathscr X_{n+1}(gKg^{-1}),\underline{M}_{\boldsymbol{\mu}}(E)),
\end{CD}
$$
for $*\in\{-,\rm c,!\}$. We use the same notation mutatis mutandis for $\mathscr X_{n+1}^{\ad}$.

When dealing with tensor products as coefficients, we assume $\BQ(\boldsymbol{\mu},\boldsymbol{\nu})\subseteq E$ and identify
$$
M_{\boldsymbol{\mu}\times\boldsymbol{\nu}}(E)=
M_{\boldsymbol{\mu}}(E)\otimes_E M_{\boldsymbol{\nu}}(E)
$$
and consider it as a $G_{n+1}\times G_n$-module and by abuse of notation also as a $G_n$-module via the diagonal embedding $j\times 1$, and identify
$$
\rho_{\boldsymbol{\mu}\times\boldsymbol{\nu}}(g)=
\rho_{\boldsymbol{\mu}\times\boldsymbol{\nu}}(j(g),g),
\quad g\in G_n(\BA_\BQ^{(\infty)}).
$$
In particular we have for each $g\in G_n(\BA_\BQ^{(\infty)})$, $*\in\{-,\rm c,!\}$, a map
$$
t_g^{\boldsymbol{\mu}\times\boldsymbol{\nu}}:\quad
H_*^q(\mathscr X_{n}(K'),\underline{M}_{\boldsymbol{\mu}\times\boldsymbol{\nu}}(E))\to
H_*^q(\mathscr X_{n}(gK'g^{-1}),\underline{M}_{\boldsymbol{\mu}\times\boldsymbol{\nu}}(E)).
$$

\subsection{Global sections}

The reason for the notation $\rho_{\boldsymbol{\mu}}(g)$ will become clear when considering global sections over connected components $\mathscr X_{n+1}(K)[c]$ as elements of
$$
H^0(\Gamma_c; M_{\boldsymbol{\mu}}(E))\;\subseteq\; M_{\boldsymbol{\mu}}(E)
$$
via the identification \eqref{eq:symmetriccomponent}, which depends on the choice of the representative $c\det(K)\in\Gm(\BA_k^{(\infty)})/\det(K)$ of the class $[c]\in C(K)$. Only for $[1]\in C(K)$ we have the canonical choice $c=1$, where a section
$$
f\;\in\;
H^0(\mathscr X_{n+1}(K)[1]; \underline{M}_{\boldsymbol{\mu}}(E))
$$
can be sent canonically to its evaluation (as a function) at the origin
$$
f({\bf1}_{n+1}K_{n+1}^0\times K)\;\in\;H^0(\Gamma_1; M_{\boldsymbol{\mu}}(E)).
$$
However for all other classes there are no canonical choices, i.e.\ there is no canonical identification of a section in
$$
H^0(\mathscr X_{n+1}(K)[c]; \underline{M}_{\boldsymbol{\mu}}(E))
$$
with an element in $M_{\boldsymbol{\mu}}(E)$.

This is of particular importance in the context of integral structures: If $g\in G_{n+1}(\BQ)\cap G_{n+1}(\BA_\BQ^{(\infty)})$, then sections of the sheaf associated to a lattice
$$
M_{\boldsymbol{\mu}}(\OO)\;\subseteq\;M_{\boldsymbol{\mu}}(E)
$$
will be mapped by $t_g^{\boldsymbol{\mu}}$ to sections of the sheaf associated to the translated lattice
\begin{equation}
g\cdot M_{\boldsymbol{\mu}}(\OO)=\rho_{\boldsymbol{\mu}}(g) M_{\boldsymbol{\mu}}(\OO)
\;\subseteq\;
M_{\boldsymbol{\mu}}(E),
\label{eq:sheaftranslation}
\end{equation}
which will be defined in section \ref{subsec:boundedness} when we introduce the necessary adelic formalism to handle this situation.

For the rational situation we introduce for $\gamma\in G_n(\BQ)$ the map
$$
t_{\gamma}^{\boldsymbol{\nu}}:\quad
H^0(\Gamma_{c}; M_{\boldsymbol{\nu}}(E))
\to
H^0(\gamma\Gamma_{c}\gamma^{-1}; M_{\boldsymbol{\nu}}(E)),
$$
$$
s\;\mapsto\;\rho_{\boldsymbol{\nu}}(\gamma)(s).
$$
We set for $g\in G_n(\BA_\BQ^{(\infty)})$,
$$
\Gamma_{c}^g\;:=\;
G_n(\BQ)\cap d_{(c)}gK'g^{-1}d_{(c)}^{-1},
$$
the intersection taking place in $G_n(\BA_\BQ^{(\infty)})$.

\begin{prp}\label{prop:pullbacks}
For any $c,c'\in \GL_1(\BA_k^{(\infty)})$ and any $g\in G_n(\BA_\BQ^{(\infty)})$ with $$
[c']\;=\;[\det(g)c]\;\in\;C(K')
$$
there is a $\gamma\in G_n(\BQ)$ with finite part
\begin{equation}\label{eq:gammadef}
\gamma_f\;\in\;d_{(c)}gK'd_{(c')}^{-1}.
\end{equation}
Any such element renders the diagram
$$
\begin{CD}
H^0(\Gamma_{c'}; M_{\boldsymbol{\nu}}(E))
@>t_{\gamma}^{\boldsymbol{\nu}}>>
H^0(\Gamma_{c}^g; M_{\boldsymbol{\nu}}(E))\\
@V\cong{}V\eqref{eq:symmetriccomponent}V @V\eqref{eq:symmetriccomponent}V\cong{}V\\
H^0(\mathscr X_n(K')[c']; \underline{M}_{\boldsymbol{\nu}}(E))
@>t_{g}^{\boldsymbol{\nu}}>>
H^0(\mathscr X_n(gK'g^{-1})[c]; \underline{M}_{\boldsymbol{\nu}}(E))
\end{CD}
$$
commutative.
\end{prp}
We remark that $\gamma$ depends not only on $g$ but also on the choice of the representatives $c$, $c'$.

\begin{proof}
The existence of $\gamma$ satisfying \eqref{eq:gammadef} follows by strong approximation from the identity
$$
d_{(c)}gK'd_{(c')}^{-1}
=d_{(c)}gd_{(c')}^{-1}\cdot(d_{(c')}K'd_{(c')}^{-1})
$$
as the element $d_{(c)}gd_{(c')}^{-1}$ on the right hand side has determinant in $\det(K')$ and $K'$ is open. Such a $\gamma$ satisfies
$$
\Gamma_c^{g}=G_n(\BQ)\cap \gamma d_{(c')} K'd_{(c')}^{-1}\gamma^{-1}=\gamma \Gamma_{c'}\gamma^{-1},
$$
consequently $t_{\gamma}^{\boldsymbol{\nu}}$ is well defined.

In this context the identification \eqref{eq:symmetriccomponent} is explicitly given by
$$
i_{c'}:\quad
\Gamma_{c'} g_\infty' K_n\;\mapsto\;G_n(\BQ)g_\infty' d_{(c')} K_n^0 K',\quad g_\infty'\in G_n(\BR)^0.
$$
Similarly we have the map
$$
i_c^{g}:\quad
\Gamma_c^{g} g_\infty' K_n\;\mapsto\;G_n(\BQ)g_\infty' d_{(c)} K_n^0 gK'g^{-1}.
$$
We consider the map
$$
t_\gamma':\quad
\Gamma_c^{g} g_\infty' K_n\;\mapsto\;\Gamma_{c'} \gamma_\infty^{-1}g_\infty' K_n.
$$
This yields the commutative square
$$
\begin{CD}
\mathscr X_n(K')[c']
@<t_{g}<<
\mathscr X_n(gK'g^{-1})[c]\\
@Ai_{c'}AA @AAi_c^{g}A\\
\Gamma_{c'}\backslash{}G_n(\BR)/K_n
@<t_\gamma'<<
\Gamma_{c}^{g}\backslash{}G_n(\BR)/K_n
\end{CD}
$$
where all maps are diffeomorphisms.

Thus on global sections we obtain a commutative diagram
$$
\begin{CD}
H^0(\mathscr X_n(K')[c']; \underline{M}_{\boldsymbol{\nu}}(E))
@>t_{g}^{\boldsymbol{\nu}}>>
H^0(\mathscr X_n(gK'g^{-1})[c]; \underline{M}_{\boldsymbol{\nu}}(E))\\
@V{i_{c'}^*}VV @VV{(i_c^{g})^*}V\\
H^0(\Gamma_{c'}\backslash{}G_n(\BR)/K_n;
\underline{M}_{\boldsymbol{\nu}}(E))
@>(t_\gamma')^{\boldsymbol{\nu}}>>
H^0(\Gamma_{c}^{g}\backslash{}G_n(\BR)/K_n;
\underline{M}_{\boldsymbol{\nu}}(E))\\
@V\cong{}VV @VV\cong{}V\\
H^0(\Gamma_{c'}; M_{\boldsymbol{\nu}}(E))
@>t_{\gamma}^{\boldsymbol{\nu}}>>
H^0(\Gamma_c^g; M_{\boldsymbol{\nu}}(E))\\
\end{CD}
$$
where we have to explain the precise meaning of the various pullbacks in the upper square. First of all we have the explicit global sections
$$
H^0(\Gamma_{c'}\backslash{}G_n(\BR)/K_n;
\underline{M}_{\boldsymbol{\nu}}(E))=
\{f:
G_n(\BR)/K_n\to M_{\boldsymbol{\nu}}(E)\mid
$$
$$
f'\;\text{constant and}\;
\forall \gamma'\in\Gamma_{c'}:\;f'(\gamma' g_\infty' K_n)=\rho_{\boldsymbol{\nu}}(\gamma') f'(g_\infty' K_n)\},
$$
the map $i_{c'}^*$ being the usual pullback of maps. The same holds mutatis mutandis for $(i_c^{g})^*$ and $(t_\gamma')^*$. Then
$$
(t_\gamma')^{\boldsymbol{\nu}}:\quad
f'\;\mapsto\;\rho_{\boldsymbol{\nu}}(\gamma)(t_\gamma')^*f'\;=\;[u'\mapsto \rho_{\boldsymbol{\nu}}(\gamma)f'(\gamma_\infty^{-1}u')].
$$
The natural isomorphism
$$
H^0(\Gamma_{c'}\backslash{}G_n(\BR)/K_n;\underline{M}_{\boldsymbol{\nu}}(E))\to
H^0(\Gamma_{c'};M_{\boldsymbol{\nu}}(E))
$$
is given by evaluation of $f'$ at the origin ${\bf1}_n K_n$. Thus we see that
$t_{\gamma}^{\boldsymbol{\nu}}$ is given by
$$
f'({\bf1}_nK_n)
\;\mapsto\;
\rho_{\boldsymbol{\nu}}(\gamma)f'({\bf1}_n K_n).
$$
This concludes the proof.
\end{proof}

\begin{corp}\label{cor:pullbacks}
For any $c,c'\in\GL_1(\BA_k^{(\infty)})$ and $g\in G_n(\BA_\BQ^{(\infty)})$ satisfying
\begin{equation}\label{eq:pullbackidcond}
c'\in c\det(gK')
\end{equation}
we have a commutative square
$$
\begin{CD}
H^0(\mathscr X_n(K')[c']; \underline{M}_{\boldsymbol{\nu}}(E))
@>t_{g}^{\boldsymbol{\nu}}>>
H^0(\mathscr X_n(gK'g^{-1})[c]; \underline{M}_{\boldsymbol{\nu}}(E))\\
@V\eqref{eq:symmetriccomponent}VV @VV\eqref{eq:symmetriccomponent}V\\
H^0(\Gamma_{c'}; M_{\boldsymbol{\nu}}(E))
@=
H^0(\Gamma_{c}^g; M_{\boldsymbol{\nu}}(E))
\end{CD}
$$
\end{corp}

\begin{proof}
For $\gamma$ as in Proposition \ref{prop:pullbacks} we may assume $\gamma\in G_n^{\der}(\BQ)$ by \eqref{eq:pullbackidcond}, then $t_{\gamma}^{\boldsymbol{\nu}}$ is the identity, map
$$
H^0(\Gamma_{c'}; M_{\boldsymbol{\nu}}(E))\to H^0(\Gamma_{c}^g; M_{\boldsymbol{\nu}}(E)),
$$
as by Lemma \ref{lem:gammannormal} $G_n^{\der}\subseteq \Gamma_n$, the common algebraic closure of $\Gamma_{c'}$ and $\Gamma_{c}^g$, thus $\rho_{\boldsymbol{\mu}}(\gamma)$ acts trivially on the left hand side.
\end{proof}

\subsection{Fundamental classes}

We write $H_q^{BM}(X,\BZ)$ for the Borel-Moore homology of a CW-complex $X$ with coefficients in $\BZ$. For $K\subseteq G_n(\BA_\BQ^{(\infty)})$ neat we know that $\mathscr X_{n}(K)[c]$ is orientable for each $c\in C(K)$. Thus we may fix a fundamental class
$$
F_{K,c}\;\in\;
H_{\dim\mathscr X_{n}(K)}^{BM}(\mathscr X_{n}(K)[c],\BZ)\cong\BZ.
$$
We assume $F_{K,c}$ to be chosen for all pairs $(K,c)$, $K\subseteq G_{n}(\BA_\BQ^{(\infty)})$ a neat compact open and every $c\in C(K)$, enjoying the following compatbility relations:
\begin{itemize}
\item[(i)]
For all neat compact open $K$, $c\in C(K)$ and $g\in G_{n}(\BA_\BQ^{(\infty)})$ we have
\begin{equation}
t_g^*(F_{K,c})=F_{gKg^{-1},[\det(g)c]}.
\label{eq:translationinvariance}
\end{equation}
\item[(ii)]
For all neat compact open $K\subseteq K'\subseteq G_{n}(\BA_\BQ^{(\infty)})$, $c\in C(K),c'\in C(K')$ such that $c'|_K=c$ we have
\begin{equation}
{\rm proj}_{K\to K',*}(F_{K,[c]})=\frac{(K':K)}{(C(K'):C(K))}\cdot F_{K',[c']},
\label{eq:restrictioninvariance}
\end{equation}
where ${\rm proj}_{K\to K'}$ denotes the projection map $\mathscr X_{n}(K)[c]\to\mathscr X_{n}(K')[c']$.
\end{itemize}
That such a choice satisfying (i) and (ii) exists follows from the fact that $G_{n}(\BA_\BQ)$ admits a Haar measure.

Now $F_{K,c}$ induces a unique isomorphism
$$
H_{\rm c}^{\dim\mathscr X_{n}(K)}(\mathscr X_{n}(K)[c],\BZ)\to\BZ,
$$
and by Poincar\'e duality provides us with a perfect pairing
$$
H^0(\mathscr X_{n}(K)[c],\underline{M}_{\boldsymbol{\mu}^\vee\times\boldsymbol{\nu}^\vee}(E))\otimes_E
H_{\rm c}^{\dim\mathscr X_{n}(K)}(\mathscr X_{n}(K)[c],\underline{M}_{\boldsymbol{\mu}\times\boldsymbol{\nu}}(E))
\to E,
$$
that we interpret as an isomorphism
$$
\int\limits_{\mathscr X_{n}(K)[c]}:
H_{\rm c}^{\dim\mathscr X_{n}(K)}(\mathscr X_{n}(K)[c],\underline{M}_{\boldsymbol{\mu}\times\boldsymbol{\nu}}(E))
\to
H^0(\mathscr X_{n}(K)[c],\underline{M}_{\boldsymbol{\mu}\times\boldsymbol{\nu}}(E)).
$$
If we interpret a cohomology class on the left hand side as an $M_{\boldsymbol{\mu}\times\boldsymbol{\nu}}(\BC)$-valued differential form $\omega$, then the image of this class under the above map is given by integration of $\omega$ over the manifold $\mathscr X_{n}(K)[c]$ with respect to the orientation $F_{K,c}$. Writing $x_1,\dots,x_{\dim\mathscr X_{n}(K)}$ for a Maurer-Cartan basis of $\g_n/(\k_n+\s_n)$, we may think of $\omega$ explicitly as
$$
\omega\;=\;
\rho_{\boldsymbol{\mu}\times\boldsymbol{\nu}}(g_\infty)\varphi(g)
\cdot \mathrm{d}x_1\wedge\cdots\wedge \mathrm{d}x_{\dim\mathscr X_{n}(K)},
$$
where
$$
\varphi:\quad\mathscr X_n(K)[c]\to\
M_{\boldsymbol{\mu}\times\boldsymbol{\nu}}(\BC)
$$
is a compactly supported smooth function. For the details of this correspondence we refer to section 3.5 of \cite{Jan14}. We obtain
$$
\int\limits_{\mathscr X_{n}(K)[c]}\omega
=
\int 
\rho_{\boldsymbol{\mu}\times\boldsymbol{\nu}}(g_\infty)\varphi(g)
\cdot \mathrm{d} x_1\wedge\cdots\wedge \mathrm{d}x_{\dim\mathscr X_{n}(K)}
\in
H^0(\Gamma_c; M_{\boldsymbol{\mu}\times\boldsymbol{\nu}}(\BC)).
$$
As $\Gamma_c$ is Zariski-dense in $\Gamma_n\subseteq G_n$, the projection $\xi_j$ induces a map
$$
\xi_j:\quad
H^0(\Gamma_c; M_{\boldsymbol{\mu}\times\boldsymbol{\nu}}(\BC))
\;\to\;
M_{(j)}(\BC),
$$
which indeed is defined over $E$ (even $\OO$). Then for a suitable choice of $\varphi$, application of $\xi_j$ (and summing over all $c\in C(K)$), turns the above integral into the the Rankin-Selberg integral over $\varphi\otimes\sgn^j$ (as a function on $G_n(\BA_\BQ)$) at $s=\frac{1}{2}+j$, because $\xi_j$ commutes with integration.

\subsection{Hecke operators on cohomology}

For a compact open $K'\subseteq G_{n}(\BA_\BQ^{(\infty)})$ the compact open double coset $K'gK'$ represented by an element $g\in G_{n}(\BA_\BQ^{(\infty)})$ decomposes into finitely many open right cosets
$$
K'gK'=\bigsqcup_i g_iK'.
$$
Then $K'gK'$ acts on sections
$$
s\in\Gamma(U;\underline{M}_{\boldsymbol{\nu}}(E))
$$
via
$$
s|_{[K'gK']}\;:=\;\sum_{i}t_{g_i}^{\boldsymbol{\nu}}(s)
$$
again on the space $\mathscr X_n(K')$. This action extends to cohomology, i.e.\ we have
$$
\cdot|_{[K'gK']}\in\End_E(H^q_*(\mathscr X_{n}(K');\underline{M}_{\boldsymbol\nu}(E)))
$$
for any $*\in\{-,\rm c,!\}$. Note that this action is compatible with the Hecke action on group cohomology as well as with the action on automorphic forms.

\section{Cohomological construction of the distribution}\label{sec:distribution}

\subsection{The rational modular symbol}\label{subsec:rationalmodularsymbol}

%We give a quick sketch of the construction of our distribution ignoring integrality questions, postponing the treatment of integrality.

For any compact open subgroups $K\subseteq G_{n+1}(\BA_\BQ^{(\infty)})$, $K'\subseteq G_{n}(\BA_\BQ^{(\infty)})$, we consider the proper \lq{}diagonal\rq{} map
$$
j\times\ad:\quad
\mathscr X_{n}(j^{-1}(K)\cap K')
\to
\mathscr X_{n+1}^{\ad}(K)\times\mathscr X_{n}^{\ad}(K'),
$$
which induces a map
$$
(j\times\ad)^*:\quad
H_{\rm c}^{b_{n+1}^k+b_{n}^k}(\mathscr X_{n+1}^{\ad}(K)\times\mathscr X_{n}^{\ad}(K'); \underline{M}_{\boldsymbol{\mu}\times\boldsymbol{\nu}}(E))
$$
$$
\to
H_{\rm c}^{b_{n+1}^k+b_{n}^k}(\mathscr X_{n}(j^{-1}(K)\cap K'); \underline{M}_{\boldsymbol{\mu}\times\boldsymbol{\nu}}(E)),
$$
where we suppress the pullback on the coefficients in the notation,
$\underline{M}_{\boldsymbol{\mu}\times\boldsymbol{\nu}}(E))$ on the right hand side being the tautological sheaf on $\mathscr X_{n}(j^{-1}(K)\cap K')$ associated to the $G_n$-module $(j\times 1)^*M_{\boldsymbol{\mu}\times\boldsymbol{\nu}}(E)$.

We remark that
$$
b_{n+1}^k+b_{n}^k \;=\; \dim \mathscr X_{n}(j^{-1}(K)\cap K'),
$$
thus the pullback along $j\times\ad$ produces a class in top degree, which may be integrated against our chosen fundamental class.

We define for $h'\in G_{n}(\BA_\BQ^{(\infty)})$ and $h\in G_{n+1}(\BA_\BQ^{(\infty)})$,
$$
K(h,h'):=
j^{-1}(hKh^{-1})\cap h'K'(h')^{-1}
\;\subseteq\;
G_{n}(\BA_\BQ^{(\infty)}),
$$
and for every representative $c$ of a class in $C(K(h,h'))$ the arithmetic group
$$
\Gamma_c^{h,h'}\;:=\;G_n(\BQ)\cap d_{(c)}K(h,h')d_{(c)}^{-1}.
$$

With this formalism at hand we are ready to define for $h,h'$ and $c\in C(K(h,h'))$ with the property that $K(h,h')$ is neat (which is automatic whenever $K$ or $K'$ is neat), our topological modular symbol as the map
$$
\mathscr P_{h,h',c}^{K,K'}:
H_{\rm c}^{b_{n+1}^k+b_{n}^k}(\mathscr X_{n+1}^{\ad}(K)\times\mathscr X_{n}^{\ad}(K'); \underline{M}_{\boldsymbol{\mu}\times\boldsymbol{\nu}}(E))
\to
H^0(\Gamma_c^{h,h'}; M_{\boldsymbol{\mu}\times\boldsymbol{\nu}}(E)),
$$
$$
\lambda\;\mapsto\;
%\!\!\!\!\!\!\!\!\!\!\!\!\!\!\!\!\!\!\!\!\!\!\!\!\!\!\!\!\!\!\!
%\int\limits_{\mathscr X_{n}(j^{-1}(hKh^{-1})\cap K')[\det(t)]}
%\int\limits_{\mathscr X_{n}(j^{-1}(hKh^{-1})^{-1}\cap h'K'(h')^{-1})[c]}
%\!\!\!\!\!\!\!\!\!\!\!\!\!\!\!\!\!\!\!\!\!\!\!\!\!
\int\limits_{\mathscr X_{n}(K(h,h'))[c]}
\!\!\!\!\!\!\!\!\!
%t^{\boldsymbol{\mu}\times\boldsymbol{\nu}}_{l}
(j\times\ad)^*\circ
%t^{\boldsymbol{\mu}\times\boldsymbol{\nu}}_{(h,{\bf1}_n)}(\lambda).
t^{\boldsymbol{\mu}\times\boldsymbol{\nu}}_{(h,h')}(\lambda).
$$
It enjoys the following elementary properties.
\begin{prp}\label{prop:modularsymbol1}
As a function of $h$ and $h'$, $\mathscr P_{h,h',c}^{K,K'}$ only depends on the right cosets
$$
hK\quad\text{and}\quad
h'K',
$$
and more precisely only on the compact open double coset
$$
(j\times 1)
%(j^{-1}(hKh^{-1})\cap h'K'(h')^{-1})
K(h,h')
\cdot (hK\times h'K')\;\subseteq\;
G_{n+1}(\BA_\BQ^{(\infty)})\times G_{n}(\BA_\BQ^{(\infty)}).
$$
\end{prp}

\begin{proof}
The constancy on the cosets $hK$ and $h'K'$ is clear as
$$
t_{(hk,h'k')}^{\boldsymbol{\mu}\times\boldsymbol{\nu}}(\lambda)
=
t_{(h,h')}^{\boldsymbol{\mu}\times\boldsymbol{\nu}}(\lambda),
$$
for $k\in K$, $k'\in K'$, and
$$
\lambda\;\in\;H_{\rm c}^{b_{n+1}^k+b_{n}^k}(\mathscr X_{n+1}^{\ad}(K)\times\mathscr X_{n}^{\ad}(K'); \underline{M}_{\boldsymbol{\mu}\times\boldsymbol{\nu}}(E)).
$$
Furthermore for any
$$
k'\;\in\;
%j^{-1}(hKh^{-1})\cap h'K'(h')^{-1}
K(h,h')
$$
we get
$$
\mathscr P_{j(k')h,k'h',c}^{K,K'}(\lambda)\;=\;
\int\limits_{\mathscr X_{n}(K(j(k')h,k'h'))[c]}
\!\!\!\!\!\!\!\!\!\!\!\!\!
%t^{\boldsymbol{\mu}\times\boldsymbol{\nu}}_{l}
t^{\boldsymbol{\mu}\times\boldsymbol{\nu}}_{k'}\circ
(j\times\ad)^*\circ
%t^{\boldsymbol{\mu}\times\boldsymbol{\nu}}_{(h,{\bf1}_n)}(\lambda).
t^{\boldsymbol{\mu}\times\boldsymbol{\nu}}_{(h,h')}(\lambda).
$$
However as
$$
t^{\boldsymbol{\mu}\times\boldsymbol{\nu}}_{k'}\circ
(j\times\ad)^*\circ
%t^{\boldsymbol{\mu}\times\boldsymbol{\nu}}_{(h,{\bf1}_n)}(\lambda).
t^{\boldsymbol{\mu}\times\boldsymbol{\nu}}_{(h,h')}(\lambda)\;=\;
(j\times\ad)^*\circ
%t^{\boldsymbol{\mu}\times\boldsymbol{\nu}}_{(h,{\bf1}_n)}(\lambda).
t^{\boldsymbol{\mu}\times\boldsymbol{\nu}}_{(h,h')}(\lambda),
$$
and as conjugation by $k'$ fixes the compact open groups in question, we conclude that
$$
\mathscr P_{j(k')h,k'h',c}^{K,K'}(\lambda)\;=\;
\mathscr P_{h,h',c}^{K,K'}(\lambda),
$$
as claimed.
\end{proof}

More generally we have
\begin{prp}\label{prop:modularsymbol2}
For any $h$, $h'$ and $c$ as above, and $g'\in G_n(\BA_\BQ^{(\infty)})$ we have
$$
\mathscr P_{j(g')h,g'h',c}^{K,K'}\;=\;
%t_{g'}^{\boldsymbol{\mu}\times\boldsymbol{\nu}}\circ
\mathscr P_{h,h',\det(g')c}^{K,K'}.
$$
\end{prp}

Here and in the sequel such an identity is understood via the identification
$$
H^0(\Gamma_{\det(g')c}^{h,h'}; M_{\boldsymbol{\mu}\times\boldsymbol{\nu}}(E))
=
H^0(\Gamma_c^{j(g')h,g'h'}; M_{\boldsymbol{\mu}\times\boldsymbol{\nu}}(E))
$$
of the two domains as identical subspaces of the representation space $M_{\boldsymbol{\mu}\times\boldsymbol{\nu}}(E)$.
\begin{proof}
Replacing $k'$ in the second part of the proof of Proposition \ref{prop:modularsymbol1} by $g'$ the claim follows mutatis mutandis from Corollary \ref{cor:pullbacks} and the relation
$$
t^{\boldsymbol{\mu}\times\boldsymbol{\nu}}_{g'}\circ\!\!\!
\int\limits_{\mathscr X_{n}(K(h,h'))[\det(g)c]}
\;=\;
\int\limits_{\mathscr X_{n}(K(j(g')h,g'h'))[c]}\!\!\!\!\circ
\;
t^{\boldsymbol{\mu}\times\boldsymbol{\nu}}_{g'}
$$
which itself is a consequenc of \eqref{eq:translationinvariance}.
\end{proof}

\subsection{Rational construction of distributions}

We fix once and for all an element
$$
\varpi\;\in\;
\OO_p:=\OO\otimes_\BZ \BZ_p\;\subseteq\;
\BA_\BQ^{(\infty)}
$$
with $v_\mathfrak{p}(\varpi)\geq m > 0$ for all $\mathfrak{p}\mid p$. We introduce the matrices
$$
t_{(p)}'\;:=\;
\begin{pmatrix}
\varpi^n&&&\\
&\varpi^{n-1}&&\\
&&\ddots&\\
&&&\varpi
\end{pmatrix},
$$
in $\GL_n(k)$ and
$$
t_{(p)}:=j(t_{(p)}')=
\begin{pmatrix}
\varpi^n&&&\\
&\varpi^{n-1}&&\\
&&\ddots&\\
&&&1
\end{pmatrix},
$$
which give rise to a Hecke operator
$$
U_p:= K_p t_{(p)} K_p\otimes K_p' t_{(p)}' K_p',
$$
where $K_p$ and $K_p'$ are the products of the mod $\mathfrak{p}^m$ Iwahori subgroups $K_{\mathfrak{p}}(m)$ resp.\ $K_{\mathfrak{p}}'(M)$ for $\mathfrak{p}\mid p$. Then $U_p$ is product of the operators
$$
V_{\mathfrak{p}}\otimes V_{\mathfrak{p}}',\quad\mathfrak{p}\mid p,
$$
thus conditions like ordinarity and finite slope translate from one operator to the other, i.e.\ the appropriate notions are equivalent, see below.

We assume that we are given neat compact open $K$ and $K'$ satisfying $\det(K)=\det(K')$ and having $p$-components $K_p$ resp.\ $K_p'$, in particular the operator $U_p$ then acts on the cohomology of the manifold $\mathscr X_{n+1}(K)\times\mathscr X_n(K')$. We assume we are given an eigen function
$$
\lambda\;\in\;
H_{\rm c}^{b_{n+1}^k+b_{n}^k}(\mathscr X_{n+1}^{\ad}(K)\times\mathscr X_{n}^{\ad}(K'); \underline{M}_{\boldsymbol{\mu}\times\boldsymbol{\nu}}(E))
%\otimes_E
%H_{\rm c}^{b_{n}^k}(\mathscr X_{n}^{\ad}(K'); \underline{M}_{\boldsymbol{\nu}}(E))
$$
with non-zero eigen value $\kappa\in E^\times$.

Let $\mathfrak{f}\subseteq\OO_p$ be a proper ideal which is generated by a power $\varpi^v$, $v\geq 1$. Any class in $C(\mathfrak{f})$ may be represented by an element $x\in\BA_\BQ^{(\infty)}$ and we set
$$
\tilde{\mu}_\lambda(x+\mathfrak{f})\,:=\,
\kappa^{-v}\cdot
\mathscr P_{h^{(1)}t_{(p)}^v,(t_{(p)}')^v,x}^{K,K'}(\lambda)
\;\in\;H^0(
\Gamma_x^{(v)};
M_{\boldsymbol{\mu}\times\boldsymbol{\nu}}(E)),
$$
where
$$
\Gamma_x^{(v)}\;:=\;\Gamma_x^{h^{(1)}t_{(p)}^v,(t_{(p)}')^v}.
$$
Finally we set
$$
\mu_\lambda(x+\mathfrak{f})\,:=\,
\xi\circ\tilde{\mu}_\lambda(x+\mathfrak{f}).
$$
Note that the notation $\tilde{\mu}_\lambda(x+\mathfrak{f})$ suggests independence of the representative $x$, which is only true up to natural isomorphisms in the codomain. Strictly speaking the codomain depends on the coset $x\det(K(\varpi^v))$, where
$$
K(\varpi^v)\;:=\;K(h^{(1)}t_{(p)}^v,(t_{(p)}')^v).
$$
This has to be kept in mind in the sequel. However, to be clear, all identities are understood as strict identities in what follows (and not just identical up to natural isomorphism).

We know that the $p$-component of $\det(K(\varpi^v))$ equals $1+(\varpi^v)$, cf.\ \cite[Proposition 3.4]{Sch2}. Therefore we have a finite covering map
\begin{equation}
C(K(\varpi^v))\to C((\varpi^v)).
\label{eq:Cfcover}
\end{equation}

In the absence of complex places $\xi$ is an isomorphism, otherwise, depending on the weight $\boldsymbol{\mu}\times\boldsymbol{\nu}$, it is not. As of now it is unclear if $\tilde{\mu}_\lambda$ carries more arithmetic information than $\mu_\lambda$, and which arithmetic phenomenon this information may reflect.

\subsection{The distribution relation}

We import
\begin{lem}\label{lem:distributionmatrixrelation}\cite[Lemma 4.1]{Jan14}
For any $u\in U_{n+1}(\OO_p)$, and $w\in U_{n}(\OO_p)$, there exist matrices $k_{u,w}'\in K_p'$ and $k_{u,w}\in K_p$ with the property that
$$
t_{(p)}^{-1} j(w)^{-1}\cdot t_{(p)}^{-v}h^{(1)}t_{(p)}^v\cdot u t_{(p)}\;=\;
j(k_{u,w}')\cdot t_{(p)}^{-(v+1)} h^{(1)}t_{(p)}^{v+1}\cdot k_{u,w}
$$
and that furthermore sending $u,w$ to
$$
\det(k_{u,w})=\det(k_{u,w}')^{-1}\pmod{(\varpi^{v+1})}
$$
defines an epimorphism of groups
$$
U_{n+1}(\OO_p)/t_{(p)}U_{n+1}(\OO_p)t_{(p)}^{-1}\times
U_{n}(\OO_p)/t_{(p)'}U_{n}(\OO_p)(t_{(p)}')^{-1}
$$
$$
\to
(1+(\varpi^v))/(1+(\varpi^{v+1})).
$$
\end{lem}

\begin{proof}
Applying \cite[Lemma 4.1]{Jan14} to each place dividing $p$ yields iteratively the existence of $k_{u,w}'$ and $k_{u,w}$ satisfying the first statement. The second statement follows from loc.\ cit.\ as well.
\end{proof}

For the distribution relation the following generalization of \cite[Lemma 4.2]{Jan14} is crucial.
\begin{lem}\label{lem:distributionintegralrelation}
For any (not necessarily Hecke eigen class)
$$
\lambda\in H^{b_{n+1}^k+b_n^k}_{\rm c}(
\mathscr X_{n+1}^{\rm ad}(K)\times\mathscr X_{n}^{\rm ad}(K');
\underline{M}_{\boldsymbol{\mu}\times\boldsymbol{\nu}}(E)),
$$
and any $u\in U_{n+1}(\OO_p)$ and $w\in U_{n}(\OO_p)$ we have
$$
\mathscr P_{h^{(1)}t_{(p)}^vut_{(p)},(t_{(p)}')^vwt_{(p)}',x}^{K, K'}
(\lambda)
=
\mathscr P_{h^{(1)}t_{(p)}^{v+1},(t_{(p)}')^{v+1},x\det(k_{u,w})}^{K,K'}
(\lambda)
$$
\end{lem}

\begin{proof}
By Proposition \ref{prop:modularsymbol2} and Lemma \ref{lem:distributionmatrixrelation} we obtain
\begin{eqnarray*}
&&
\mathscr P_{h^{(1)}t_{(p)}^vut_{(p)},(t_{(p)}')^vwt_{(p)}',x}^{K, K'}
(\lambda)\\
&=&
%t_{(t_{(p)}')^vwt_{(p)}'}^{\boldsymbol{\mu}\times\boldsymbol{\nu}}
\mathscr P_{t_{(p)}^{-1}j(w)^{-1}t_{(p)}^{-v}h^{(1)}t_{(p)}^vut_{(p)},{\bf1}_n,\det(t_{(p)})^{(v+1)}x}^{K, K'}
(\lambda)
\quad\quad\quad
\text{(by Prop.\ \ref{prop:modularsymbol2})}\\
&=&
%t_{(t_{(p)}')^vwt_{(p)}'}^{\boldsymbol{\mu}\times\boldsymbol{\nu}}\circ
\mathscr P_{j(k_{u,w}')\cdot t_{(p)}^{-(v+1)} h^{(1)}t_{(p)}^{v+1}\cdot k_{u,w},{\bf1}_n,\det(t_{(p)})^{(v+1)}x}^{K, K'}
(\lambda)
\quad\quad\,
\text{(by Le.\ \ref{lem:distributionmatrixrelation}, Prop.\ \ref{prop:modularsymbol2})}\\
&=&
%t_{(t_{(p)}')^vwt_{(p)}'k_{u,w}'\cdot (t_{(p)}')^{-(v+1)}}^{\boldsymbol{\mu}\times\boldsymbol{\nu}}\circ
\mathscr P_{h^{(1)}t_{(p)}^{v+1},t_{(p)}^{(v+1)}(k_{u,w}')^{-1},\det(k_{u,w}')\cdot x}^{K, K'}
(\lambda)
\quad\quad\quad
\quad\quad\quad
\text{(by Prop.\ \ref{prop:modularsymbol1}, Prop.\ \ref{prop:modularsymbol2})}\\
&=&
%t_{(t_{(p)}')^vwt_{(p)}'k_{u,w^{-1}}'\cdot (t_{(p)}')^{-(v+1)}}^{\boldsymbol{\mu}\times\boldsymbol{\nu}}\circ
\mathscr P_{h^{(1)}t_{(p)}^{v+1},t_{(p)}^{(v+1)},\det(k_{u,w}')\cdot x}^{K, K'}
(\lambda)
\quad\quad\quad
\quad\quad\quad
\quad\quad\quad
\text{(by Prop.\ \ref{prop:modularsymbol1})}\\
\end{eqnarray*}
This concludes the proof.
\end{proof}

\begin{thm}\label{thm:distribution}
Let $\lambda$ be an eigen class with eigen value $\kappa\in E^\times$. Then for any $x\in\BA_k^{(\infty)}$ and any $v\geq 1$ we have the relation
$$
\tilde{\mu}_\lambda(x+(\varpi^v))\;=\;
\sum_{a\!\!\!\pmod{(\varpi)}}
\tilde{\mu}_\lambda(x+a\varpi^v+(\varpi^{v+1})).
$$
In particular after fixing representatives for all classes in $C(K(\varpi^0))$, $\tilde{\mu}_\lambda$ defines a distribution on $C(K(p^\infty))=\varprojlim\limits_{v}C(K(\varpi^v))$ with values in $H^0(\Gamma_n; M_{\boldsymbol{\mu}\times\boldsymbol{\nu}}(E))$.
\end{thm}

\begin{proof}
We set for $u$ and $w$ as in Lemma \ref{lem:distributionintegralrelation} and introduce the generalized index
$$
I(u,w)\;:=\;
\frac{(C(K(h^{(1)}t_{(p)}^v,(t_{(p)}')^v)):C(K(h^{(1)}t_{(p)}^vut_{(p)},(t_{(p)}')^vwt_{(p)}')))}
{(K(h^{(1)}t_{(p)}^v,(t_{(p)}')^v):K(h^{(1)}t_{(p)}^vut_{(p)},(t_{(p)}')^vwt_{(p)}'))}
$$
By Lemma \ref{lem:distributionintegralrelation} we see that
$$
I(u,w)\;=\;
\frac{|N_{k/\BQ}(\varpi)|}
%{(\det(K(h^{(1)}t_{(p)}^v,(t_{(p)}')^v)):\det(K(h^{(1)}t_{(p)}^{v+1},(t_{(p)}')^{v+1})))}
{(K(\varpi^v):K(\varpi^{v+1}))}
$$
is independent of $u$ and $w$. Explicitly we have by \cite[Proof of Lemma 3.2]{KMS}, \cite[Lemmata 3.7 and 3.8]{Sch2},
\begin{equation}\label{eq:index}
\begin{split}
&
(K(h^{(1)}t_{(p)}^v,(t_{(p)}')^v):K(h^{(1)}t_{(p)}^{v+1},(t_{(p)}')^{v+1}))\quad=\\
&
(U_{n+1}(\OO_p):t_{(p)}U_{n+1}t_{(p)}^{-1})
(U_{n}(\OO_p):t_{(p)}'U_{n}(t_{(p)})'^{-1}).
\end{split}
\end{equation}
We conclude that
\begin{eqnarray*}
&&
\tilde{\mu}_\lambda(x+(\varpi^v))\\
&=&
\kappa^{-v}\cdot
\mathscr P_{h^{(1)}t_{(p)}^v,(t_{(p)}')^v,x}^{K,K'}(\lambda)
\quad\quad\quad\quad
\quad\quad\quad\quad
\quad\quad\quad\quad\quad\!
\text{(by definition)}\\
&=&
\kappa^{-{v+1}}\cdot
\mathscr P_{h^{(1)}t_{(p)}^v,(t_{(p)}')^v,x}^{K,K'}(U_p\cdot \lambda)
\quad\quad\quad\quad
\quad\quad\quad
\quad\quad\quad
\text{(as $\lambda$ is eigen)}\\
&=&
\kappa^{-{v+1}}\cdot
\sum_{u,w}
I(u,w)\cdot
\mathscr P_{h^{(1)}t_{(p)}^vut_{(p)},(t_{(p)}')^vwt_{(p)}',x}^{K, K'}
(\lambda)
\quad\quad\quad\,
\text{(\eqref{eq:restrictioninvariance} and def.\ of $U_p$)}\\
&=&
\kappa^{-{v+1}}\cdot
I(1,1)\cdot
\sum_{u,w}
\mathscr P_{h^{(1)}t_{(p)}^{v+1},(t_{(p)}')^{v+1},x\det(k_{u,w})}^{K, K'}
(\lambda)
\quad\quad\;\,
\text{(by Le.\ \ref{lem:distributionintegralrelation})}\\
&=&
\sum_{a}
\kappa^{-{v+1}}\cdot
\mathscr P_{h^{(1)}t_{(p)}^{v+1},(t_{(p)}')^{v+1},x+a\varpi^v}^{K, K'}
(\lambda)
\quad\quad\quad\quad
\quad\quad\quad\,
\text{(\eqref{eq:index} and Le.\ \ref{lem:distributionintegralrelation})}\\
&=&
\sum_{a\!\!\!\pmod{(\varpi)}}
\tilde{\mu}_\lambda(x+a\varpi^v+(\varpi^{v+1})).
\end{eqnarray*}
\end{proof}

\begin{corp}\label{cor:distribution}
Under the same hypotheses as in Theorem \ref{thm:distribution}, $\mu_\lambda$ is an $M_{\boldsymbol{\mu},\boldsymbol{\nu}}(E)$-valued distribution on $C(K(p^\infty))$.
\end{corp}

We remark that along the same lines we may give a finer description of the distribution $\tilde{\mu}_\lambda$, i.e.\ by taking into account each place $\mathfrak{p}\mid p$ seperately we see that
\begin{equation}\label{eq:finemutilde}
\tilde{\mu}_\lambda(x+(\varpi))\;\in\;H^0(\Gamma_x^{(1)}; M_{\boldsymbol{\mu}\times\boldsymbol{\nu}}(E))
\end{equation}
only depends on the ideal $(\varpi)$ generated by $\varpi$, and the collection of all vectors \eqref{eq:finemutilde} for varying $\varpi$ with $v_{\mathfrak{p}}(\varpi)\geq m$ then defines mutatis mutandis a distribution with the finer property that for any ideal $\mathfrak{f}\mid p^\infty$ with $v_\mathfrak{p}(\mathfrak{f})\geq m$ for all $\mathfrak{p}\mid p$, and for any $\mathfrak{p}\mid p$, we have
$$
\tilde{\mu}_\lambda(x+\mathfrak{f})\;=\;
\sum_{a\!\!\!\pmod{\mathfrak{p}}}
\tilde{\mu}_\lambda(x+af+\mathfrak{f}\mathfrak{p}).
$$
We refrained from discussing the details of this more general approach as it would have forced us to introduce more notation, even though the result being the same.

\subsection{Integral structures and comparison maps}\label{subsec:integralstructures}

Let $\OO\subseteq E$ be a subring admitting $E$ as quotient field, and containing $\OO(\boldsymbol{\mu})$. Consider the $K$-stable $\OO$-lattice
$$
M_{\boldsymbol{\mu}}(\OO)\;\subseteq\;M_{\boldsymbol{\mu}}(E).
$$
For each $g\in G_{n+1}(\BA^{(\infty)})$ we set
\begin{equation}
gM_{\boldsymbol{\mu}}(\OO)\;:=\;
M_{\boldsymbol{\mu}}(E)\cap \rho_{\boldsymbol{\mu}}(g)M_{\boldsymbol{\mu}}(\OO\otimes_\BZ\hat{\BZ}),
\label{eq:gMdef}
\end{equation}
where the intersection takes place in
$$
M_{\boldsymbol{\mu}}(\BA_E)\;=\;
M_{\boldsymbol{\mu}}(E)\otimes_\BQ\BA_\BQ^{(\infty)}.
$$
Then the subset 
\[
 \widetilde {M}_{\boldsymbol{\mu}}(\OO)\;:=
\bigsqcup_{g\in G_{n+1}(\BA_\BQ^{(\infty)})\times G_{n+1}(\R)}
g_f{M}_{\boldsymbol{\mu}}(\OO)\times \{g\}
\]
of ${M}_{\boldsymbol{\mu}}(E)\times G_{n+1}(\A_\BQ)$ is left $G_{n+1}(\BQ)$-stable and right $GK_{n+1}^0 \times K$-stable and thus gives rise to a local system
\[
 \underline{M}_{\boldsymbol{\mu}}(\OO)\;:=\;
G_{n+1}(\BQ)\backslash \widetilde {M}_{\boldsymbol{\mu}}(\OO)/(GK_{n+1}^0 \times K)
\]
on $\mathscr X_{n+1}^{\ad}(K)$ resp.\ mutatis mutandis also on $\mathscr X_{n+1}(K)$.
Its sections on $U\subseteq\mathscr X_{n+1}(K)[1]$ are explicitly given by
$$
\Gamma(U;\underline{M}_{\boldsymbol{\mu}}(\OO))=
\{f:\Gamma U\to M_{\boldsymbol{\mu}}(\OO)\mid f\;\text{locally constant and}
$$
$$
\forall\gamma\in \Gamma,u\in \Gamma U:f(\gamma u)=\rho_{\boldsymbol{\mu}}(\gamma)f(u)
\}.
$$
The sections on the other connected components may be described similarly. To be more precise, for $g\in G_{n+1}(\BQ)$ we have the identity
$$
\underline{M}_{\boldsymbol{\mu}}(\OO)[\det(g)]
\;=\;
t_{g,*}\left(\underline{gM}_{\boldsymbol{\mu}}(\OO)|_{\mathscr X_{n+1}(gKg^{-1})[1]}\right)
$$
of sheaves on $\mathscr X_{n+1}(K)[\det(g)]$, where the sheaf $\underline{gM}_{\boldsymbol{\mu}}(\OO)$ is defined mutatis mutandis, replacing $M(\OO)$ in the definition of $\underline{M}_{\boldsymbol{\mu}}(\OO)$ with $gM(\OO)$.

In the very same way we identifty the pullback of $\underline{M}_{\boldsymbol{\mu}}(\OO)$ along the translation-by-$g$ map $t_g$ with the sheaf $\underline{gM}_{\boldsymbol{\mu}}(\OO)$, which itself is naturally a subsheaf of $\underline{M}_{\boldsymbol{\mu}}(E)$ on $\mathscr X_{n+1}(gKg^{-1})$. In order to keep track of the various incarnations, we let
$$
T_g:t_g^*\underline{M}_{\boldsymbol{\mu}}(\OO)\to
\underline{gM}_{\boldsymbol{\mu}}(\OO)
$$
denote the natural isomorphism. By construction we have natural inclusions
$$
i_g:\underline{gM}_{\boldsymbol{\mu}}(\OO)\to\underline{M}_{\boldsymbol{\mu}}(E).
$$
For notational efficiency we introduce the abbreviations
$$
iT_g:=i_g\circ T_g,\;\;\;Tt_g^*:=T_g\circ t_g^*,\;\;\;
iTt_g^*:=i_g\circ T_g\circ t_g^*.
$$
Remark that
\begin{equation}
(T_x\circ t_x^*)\circ (T_y\circ t_y^*)=T_{xy}\circ t_{xy}^*
\label{eq:Ttfunctor}
\end{equation}
and similarly for $i_g$. This formalism extends to cohomology and the cup product and duality commute with $T_g$ and $t_g^*$ and $i_g$ in the appropriate way.

\subsection{The integral modular symbol}

Assume that $\OO\subseteq E$ is a subring with quotient field $E$, which contains its localization at $p$, i.e.\ $\OO_{(p)}=\OO$.

Using the formalism from the previous section we define the topological period map $\mathscr P_{h,h',c}^{K,K'}$ over $\OO\subseteq E$ mutatis mutandis as over the quotient field $E$: Starting from a class
$$
\lambda\;\in\;
H_{\rm c}^{b_{n+1}^k+b_{n}^k}(\mathscr X_{n+1}^{\ad}(K)\times\mathscr X_{n}^{\ad}(K');
\underline{M}_{\boldsymbol{\mu}\times\boldsymbol{\nu}}(\OO))
$$
whose image in
$$
i_{({\bf 1}_{n+1},{\bf1}_n)}(\lambda)\;\in\;
H_{\rm c}^{b_{n+1}^k+b_{n}^k}(\mathscr X_{n+1}^{\ad}(K)\times\mathscr X_{n}^{\ad}(K');
\underline{M}_{\boldsymbol{\mu}\times\boldsymbol{\nu}}(E))
$$
is a $U_p$-eigen vector with eigen value $\kappa\in E^\times$, we set
$$
\tilde{\mu}_\lambda(x+\mathfrak{f})\,:=\,
\tilde{\mu}_{i_{({\bf 1}_{n+1},{\bf1}_n)}(\lambda)}(x+\mathfrak{f})\,=\,
\kappa^{-v}\cdot
\mathscr P_{h^{(1)}t_{(p)}^v,(t_{(p)}')^v,x}^{K,K'}(i_{({\bf 1}_{n+1},{\bf1}_n)}(\lambda))
$$
for $\mathfrak{f}=(\varpi^v)$, $v\geq 1$.

We may consider the integral version
$$
\mathscr P_{h,h',c}^{K,K'}:
H_{\rm c}^{b_{n+1}^k+b_{n}^k}(\mathscr X_{n+1}^{\ad}(\OO)\times\mathscr X_{n}^{\ad}(K'); \underline{M}_{\boldsymbol{\mu}\times\boldsymbol{\nu}}(\OO))
\to
H^0(K(h,h')[c]; \underline{(h,h')M}_{\boldsymbol{\mu}\times\boldsymbol{\nu}}(\OO)),
$$
given by
$$
\lambda\;\mapsto\;
\int\limits_{\mathscr X_{n}(K(h,h'))[c]}
\!\!\!\!\!\!\!\!\!
(j\times\ad)^*\circ
t^{\boldsymbol{\mu}\times\boldsymbol{\nu}}_{(h,h')}(\lambda)
$$
as before. We have a natural isomorphism
$$
\underline{(h,h')M}_{\boldsymbol{\mu}\times\boldsymbol{\nu}}(\OO)
\;=\;
\underline{hM}_{\boldsymbol{\mu}}(\OO)\otimes_\OO
\underline{h'M}_{\boldsymbol{\nu}}(\OO)
$$
of sheaves (on the various spaces where we consider it), and again we identify
$$
H^0(K(h,h')[c]; \underline{(h,h')M}_{\boldsymbol{\mu}\times\boldsymbol{\nu}}(\OO))\;\cong\;
H^0(\Gamma_c^{h,h'}; d_{(c)}(h,h')M_{\boldsymbol{\mu}\times\boldsymbol{\nu}}(\OO)).
$$

Then by construction we have obtain a commutative diagram
$$
\begin{CD}
H_{\rm c}^{b_{n+1}^k+b_{n}^k}(\mathscr X_{n+1}^{\ad}(K)\times\mathscr X_{n}^{\ad}(K');
\underline{M}_{\boldsymbol{\mu}\times\boldsymbol{\nu}}(E))
@>\mathscr P_{h,h',c}^{K,K'}>>
H^0(\Gamma_c^{h,h'}; M_{\boldsymbol{\mu}\times\boldsymbol{\nu}}(E))\\
@AAA @AAA\\
H_{\rm c}^{b_{n+1}^k+b_{n}^k}(\mathscr X_{n+1}^{\ad}(K)\times\mathscr X_{n}^{\ad}(K');
\underline{M}_{\boldsymbol{\mu}\times\boldsymbol{\nu}}(\OO))
@>\mathscr P_{h,h',c}^{K,K'}>>
H^0(\Gamma_c^{h,h'}; d_{(c)}(h,h')\cdot M_{\boldsymbol{\mu}\times\boldsymbol{\nu}}(\OO))
\end{CD}
$$
where the right vertical arrow is a monomorphism, and we see that
\begin{equation}\label{eq:tildemucodomain}
\tilde{\mu}_\lambda(x+(\varpi^{v}))\;\in\;
\kappa^{-v}\cdot
H^0(\Gamma_x^{(v)}; (h^{(1)}t_{(p)}^vd_{(x)},(t_{(p)}')^vd_{(x)})\cdot M_{\boldsymbol{\mu}\times\boldsymbol{\nu}}(\OO)).
\end{equation}

\subsection{Boundedness in the ordinary case}\label{subsec:boundedness}

We assume $E/\BQ(\boldsymbol{\mu},\boldsymbol{\nu})$ to be a number field with ring of integers $\OO_E$. Fix an embedding $i_p:E\to\overline{\BQ}_p$ and write $|\cdot|_p$ for the norm on $E$ induced by $i_p$ and write $v_p$ for the corresponding valuation. We let $\OO_{E,p}\subseteq E$ denote the corresponding valuation ring.

We introduce Hida's integrally normalized Hecke operator
$$
\tilde{U}_{p}\;:=\;
\boldsymbol{\mu}(t_{(p)})\cdot\boldsymbol{\nu}(t_{(p)}')
\cdot
U_{p},
$$
assuming without loss of generality that the uniformizer $\varpi\in\OO_p$ used in the definition of $t_{(p)}$, $t_{(p)}'$, and $U_p$, is a power of the rational prime $p$. It is well known that $\tilde{U}_{p}$ acts on cohomology with $p$-integral coefficients $M_{\boldsymbol{\mu}\times\boldsymbol{\nu}}(\OO)$.

Then if $\lambda$ is a $\tilde{U}_{p}$ eigen vector for the eigen value $\tilde{\kappa}\in E$, we say that $\lambda$ is {\em ordinary at $p$} if
\begin{equation}
|\tilde{\kappa}|_p=1
%\absNorm(\mathfrak{p})^{j_{\min}\cdot\frac{n(n+1)}{2}},
\label{eq:lambdaordinarity}
\end{equation}
which is equivalent to condition (O)
%for
%$$
%\kappa_{\lambda,p}\;:=\;\frac{\tilde{\kappa}_{\lambda,p}}{\boldsymbol{\mu}(t_{(%\mathfrak{p})})\cdot\boldsymbol{\nu}(t_{(\mathfrak{p})}')}
%$$
in section \ref{sec:maintheorem} if $\lambda$ is a cohomology class associated to $(\pi,\sigma)$. Generalizing (H), we say that $\lambda$ is {\em of finite slope} at $p$ if $\tilde{\kappa}\neq 0$.

%We remark that $M_{\boldsymbol{\mu}}(\OO_{E,p})$ and $M_{\boldsymbol{\nu}}(\OO_{E,p})$ are stable under $G_{n+1}(\BA_\BQ^{(\infty p)})$ resp.\ $G_{n}(\BA_\BQ^{(\infty p)})$.

\begin{thm}\label{thm:boundedness}
If $\lambda$ is ordinary at $p$ and is a $p$-integral cohomology class, then $\mu_\lambda$ takes values in
$
M_{(\boldsymbol{\mu},\boldsymbol{\nu})}(\OO_{E,p}),
$
i.e.\ $\mu_\lambda$ is $p$-adically bounded and thus a $p$-adic measure.
\end{thm}

\begin{proof}
We decompose $M_{\boldsymbol{\mu}\times\boldsymbol{\nu}}$ into weight spaces, i.e.
$$
M_{\boldsymbol{\mu}\times\boldsymbol{\nu}}(E)\;=\;
\bigoplus_{\eta\in X(\res_{k/\BQ}T_{n+1}\times T_n)} E_\eta,
$$
where $E_\eta$ denotes the weight space for the weight $\eta$. As $\boldsymbol{\mu}\times\boldsymbol{\nu}$ is the heighest weight for our choice of standard Borel $B=B_{n+1}\times B_n$, we know that on each $E_\eta$ with $E_\eta\neq 0$ the $p$-integral element
$$
a(v)\;:=\;((t_{(p)})^vd_{(x)},(t_{(p)}')^vd_{(x)}))
$$
acts as the scalar $\eta(a(v))\in E^\times$ with $p$-adic absolute value
$$
|\eta(a(v))|_p\;\leq\;
|{\boldsymbol{\mu}\times\boldsymbol{\nu}}(a(v))|_p.
$$
In particular we conclude with \eqref{eq:tildemucodomain} that, as $\tilde{\kappa}$ is a $p$-adic unit, 
$$
\tilde{\mu}_\lambda(x+(\varpi^{v}))\;\in\;
H^0(\Gamma_x^{(v)}; (h^{(1)},{\bf1}_n)\cdot M_{\boldsymbol{\mu}\times\boldsymbol{\nu}}(\OO)).
$$
The right hand side is a lattice independent of $v$.
\end{proof}

We remark that, in the case of positive finite slope, the same proof yields an explicit bound on the order of the resulting distribution.

\section{$p$-adic $L$-functions for Rankin-Selberg convolutions}\label{sec:padicl}

\subsection{The interpolation formula}
Let $\pi$ and $\sigma$ denote irreducible cuspidal regular algebraic automorphic representations as before. We assume they possess non-zero $K_{\mathfrak{p}}(m)$- resp.\ $K_{\mathfrak{p}}'(m)$-invariant vectors at all $\mathfrak{p}\mid p$. For each $\mathfrak{p}\mid p$ we consider the Hecke polynomial
$$
H_\mathfrak{p}(X):=\sum_{\nu=0}^{n+1} (-1)^\nu \absNorm(\mathfrak{p})^{\frac{(\nu-1)\nu}{2}}T_{\mathfrak{p},\nu} X^{n+1-\nu}
$$
in the Iwahori Hecke algebra. We choose $n$ Hecke roots
$$
\lambda_{\mathfrak{p},1},\dots,\lambda_{\mathfrak{p},n}\in E,
$$
for $\pi_\mathfrak{p}$, i.e.\ the collection of operators $H_\mathfrak{p}(\lambda_{\mathfrak{p},1}),\dots,H_{\mathfrak{p}}(\lambda_{\mathfrak{p},n})$ annihilates a non-zero vector $w_\mathfrak{p}^0$ in the Whittaker model $\mathscr{W}(\pi_\mathfrak{p},\psi_\mathfrak{p})$. Similarly we choose Hecke roots
$$
\lambda_{\mathfrak{p},1}',\dots,\lambda_{\mathfrak{p},n}'\in E,
$$
annihilating a vector
$$
0\neq v_\mathfrak{p}^0\in\mathscr{W}(\sigma_{\mathfrak{p}},\psi_\mathfrak{p}^{-1}).
$$
Note that we still choose $n$ roots on the smaller group. We set
$$
\underline{\lambda}_{\mathfrak{p}}:=
(\lambda_{\mathfrak{p},1},\dots,\lambda_{\mathfrak{p},n},
\lambda_{\mathfrak{p},1}',\dots,\lambda_{\mathfrak{p},n}')
\in E^{2n},
$$
%We set
%$$
%\lambda'':=(\lambda_{\mathfrak{p},1}',\dots,\lambda_{\mathfrak{p},n-1}').
%$$
and
$$
\kappa_{\underline{\lambda}_{\mathfrak{p}}}:=
\absNorm(\mathfrak{p})^{-\frac{(n+1)n(n-1)}{3}}\cdot
\left(
\prod_{\nu=1}^{n}
\lambda_{\mathfrak{p},\nu}^{n+1-\nu}
\right)
\cdot
\left(
\prod_{\nu=1}^{n}
\lambda_{\mathfrak{p},\nu}'^{n+1-\nu}
\right).
$$
Associated to this data we have the projection operator
$$
\Pi_{\underline{\lambda}_{\mathfrak{p}}}^0\;:=\;
\left(\prod_{i=1}^{n}
\prod_{\begin{subarray}cj=1\\j\neq i\end{subarray}}^{n+1}
(\lambda_i\absNorm(\mathfrak{p})^{1-j}T_{\mathfrak{p},j-1}-T_{\mathfrak{p},j})\right)
\otimes
\left(\prod_{i=1}^{n-1}
\prod_{\begin{subarray}cj=1\\j\neq i\end{subarray}}^{n}
(\lambda_i'\absNorm(\mathfrak{p})^{1-j}T_{\mathfrak{p},j-1}'-T_{\mathfrak{p},j}')\right)
$$

We call the tuple $(\pi,\sigma,(\underline{\lambda}_{\mathfrak{p}})_{\mathfrak{p}\mid p})$ {\em of finite slope at} $\mathfrak{p}$ if the following three conditions hold:
\begin{itemize}
\item[(i)]
$T_{\mathfrak{p},n}'$ acts on $\sigma_\mathfrak{p}$ via the scalar
\begin{equation}
\eta_{n}=\absNorm(\mathfrak{p})^{-\frac{n(n-1)}{2}}
\cdot
\prod_{\nu=1}^{n}\lambda_\nu'.
\label{eq:zentralscalar}
\end{equation}
\item[(ii)]
The vectors $w_\mathfrak{p}^0$ and $v_\mathfrak{p}^0$ may be chosen in such a way that
$$
\Pi_{\underline{\lambda}}^0(
w_\mathfrak{p}^0\otimes v_\mathfrak{p}^0)({\bf1}_{n+1},{\bf1}_{n})=
\prod_{\nu=1}^{n}\left(1-\absNorm(\mathfrak{p})^{-\nu}\right).
$$
\item[(iii)] The {\em slope}
$$
v_\mathfrak{p}\left(
\frac{\kappa_{\underline{\lambda}_{\mathfrak{p}}}}
{\boldsymbol{\mu}(a_{(\mathfrak{p})})\times\boldsymbol{\nu}(a_{(\mathfrak{p})}')}
\right)\in\BQ\cup\{\infty\},
$$
is finite.
\end{itemize}
If in addition the slope is $0$, we call the datum {\em ordinary} at $\mathfrak{p}$.

Assuming that our Whittaker vectors satisfy condition (ii), we set for each $\mathfrak{p}\mid p$,
$$
t_\mathfrak{p}:=\Pi_{\underline{\lambda}_\mathfrak{p}}^0(w_\mathfrak{p}^0\otimes v_\mathfrak{p}^0).
$$
The following Theorem is the main result towards the interpolation formula.
\begin{thm}\label{thm:localbirch}\cite[Theorem 2.3]{Jan14}
For all characters $\chi_{\mathfrak{p}}:k_{\mathfrak{p}}^\times\to\BC^\times$ with non-trivial conductor $ \mathfrak{f}_{\chi}=\OO_\mathfrak{p}\cdot f_{\chi}$, and all $0\neq f\in\OO_{\mathfrak{p}}$ with $\varpi^m\mid f$ and $f_{\chi}\mid f$ and all $s\in\BC$ we have
$$
\int\limits_{U_{n}(k_{\mathfrak{p}})\backslash{}\GL_{n}(k_{\mathfrak{p}})}
t_{\mathfrak{p}}\left(
j(g)\cdot t_{(ff_\chi^{-1})}
\cdot h^{(f)}
,
g\cdot ff_\chi^{-1}t_{(ff_\chi^{-1})}
\right)
\chi(\det(g))
|\det(g)|_{\mathfrak{p}}^{s-\frac{1}{2}}dg=
$$
$$
\prod_{\nu=1}^{n}
\left({1-\absNorm(\mathfrak{p})^{-\nu}}\right)^{-1}
\cdot
\absNorm(\mathfrak{f})^{-\frac{(n+1)n(n-1)}{6}}\cdot
\absNorm(\mathfrak{f}_\chi)^{-\frac{n(n+1)}{2}}\cdot
(\chi(f_\chi)G(\chi))^{\frac{n(n+1)}{2}}\cdot
$$
$$
t_{\mathfrak{p}}(t_{(ff_\chi^{-1})},
ff_\chi^{-1}\cdot t_{(ff_\chi^{-1})}).
$$
\end{thm}

At all finite places $\mathfrak{q}\nmid p$ we choose a good tensor
$$
t_{\mathfrak{q}}\in
\mathscr{W}(\pi_\mathfrak{q},\psi_\mathfrak{q})\otimes_\BC
\mathscr{W}(\sigma_\mathfrak{q},\psi_\mathfrak{q}^{-1}),
$$
i.e.\ a vector such that the local zeta integral computes the local $L$-function, i.e.\
$$
\int_{U_n(k_\mathfrak{q})\backslash\GL_n(k_\mathfrak{q})}
t_{\mathfrak{q}}((j\times 1)(g))|\det(g)|_\mathfrak{q}^{s-\frac{1}{2}}dg\;=\;L(s,\pi_\mathfrak{q}\times\sigma_\mathfrak{q}).
$$
At infinity, we choose
$$
t_\infty^\pm\;\in\;
\mathscr{W}(\pi_\infty,\psi_\infty)\widehat{\otimes}
\mathscr{W}(\sigma_\infty,\psi_\infty^{-1}),
$$
such that for each character $\varepsilon:\pi_0(C)\to\BC^\times$, the global cohomology class
$$
\lambda^\varepsilon\;\in\;
H^{b_{n+1}^k+b_n^k}_{\rm c}(
\mathscr X_{n+1}^{\rm ad}(K)\times\mathscr X_{n}^{\rm ad}(K'); 
\underline{M}_{\boldsymbol{\mu}\times\boldsymbol{\nu}}(\BC))
$$
associated to the global tensor
$$
t^\varepsilon:=
t_\infty^{\varepsilon}\otimes\left(\otimes_{v\nmid\infty} t_v\right)\;\in\;
\mathscr{W}(\pi,\psi)\widehat{\otimes}
\mathscr{W}(\sigma,\psi^{-1}),
$$
lies $p$-adically maximally in the $\varepsilon$-eigen space of the the image of
$$
H^{b_{n+1}^k+b_n^k}_{\rm c}(
\mathscr X_{n+1}^{\rm ad}(K)\times\mathscr X_{n}^{\rm ad}(K'); 
\underline{M}_{\boldsymbol{\mu}\times\boldsymbol{\nu}}(\OO_{E,p})).
$$
By construction we know that $t_\infty$ eventually lies in the subspace of $K\times K'$-finite vectors. For details about the construction of such cohomology classes we refer to \cite[section 3]{Jan14}. We set
$$
t:=\sum_{\varepsilon}t^\varepsilon,
$$
which has as associated cohomology class
$$
\lambda\;:=\;\sum_{\varepsilon}\lambda^\varepsilon.
$$
It is an eigen vector of each $U_{\mathfrak{p}}$ with eigen value $\kappa_{\underline{\lambda}_\mathfrak{p}}$ (cf.\ \cite[Proposition 1.3]{Jan14}). Then the associated distribution $\mu_\lambda$ computes the special values of the twisted Rankin-Selberg $L$-function, i.e.\ we have

\begin{thm}\label{thm:interpolation}
Assume that $(\pi,\sigma,\underline{\lambda})$ is of finite slope. Then there is a entire complex analytic function $\Omega_t(s)$ such that for any character $\chi:k^\times\backslash\BA_k^\times\to\overline{\BQ}^\times$ of finite order with conductor $\mathfrak{f}_\chi\mid p^\infty$, we have the interpolation formula
$$
\int\limits_{C(p^\infty)}
\chi d\mu_{\lambda}\;=\;
$$
$$
\left(
\Omega_t(\frac{1}{2}+j)\cdot
c(\chi,\frac{1}{2}+j)\cdot
L^{(p)}(\frac{1}{2}+j,(\pi\times\sigma)\otimes\chi)
\right)_{j\,\mathrm{critical\,for}\,\boldsymbol{\mu}\times\boldsymbol{\nu}},
$$
with $c(\chi,-)$ as in Theorem \ref{main:padicl}.
\end{thm}

We will discuss of the non-vanishing of the periods $\Omega_t(\frac{1}{2}+j)$ in section \ref{sec:cohreps}.

As already mentioned in section \ref{sec:maintheorem}, and as will be clear from its proof, Theorem \ref{thm:interpolation} generalizes to arbitrary finite order characters of $C(K(p^\infty))$ via \eqref{eq:Cfcover}, up to the computation of the Euler factors at the finite places $v\nmid p$ where $\det(K)\neq\OO_v^\times$. We remark that in the ordinary case the $p$-adic measure is uniquely determined by the evaluation at sufficiently ramified characters (cf.\ \cite{Jan14}).

\begin{proof}
The proof proceeds as the proof of Theorem 4.5 in \cite{Jan14}. The main ingredient being Theorem \ref{thm:localbirch} for the computation of $c(\chi_{\mathfrak{p}},-)$ in the ramified case. The computation of the unramified Euler factor for $n=1$ is standard.
\end{proof}

\subsection{The functional equation}

We have the twisted main involution
$$
\iota:g\mapsto w_{n}g^{-{\rm t}}w_{n}
$$
of $\GL_{n}$, where the supscript $-{\rm t}$ denotes matrix inversion composed with transpose, yielding an outer automorphism of $\GL_{n}$ order $2$. Let $\mathcal M$ be rational $\GL_{n}$-representation. We identify the contragredient $\mathcal M^\vee$ with the pullback of $\mathcal M$ along $\iota$. In particular it induces a twisted map
$$
\cdot^\vee:\;\;\;\mathcal M\to\mathcal M^\vee.
$$
This notion stabilizes our Hecke algebras, thus descends to Hecke modules $\mathcal M$, and gives the relation
\begin{equation}
\left({T_{\mathfrak{p},\nu}'}m\right)^\vee\;=\;
T_{\mathfrak{p},n}'\left({T_{\mathfrak{p},n-\nu}'}m\right)^\vee,\;\;\;m\in\mathcal M.
\label{eq:contragredienthecke}
\end{equation}

Now assume we are in the finite slope setting of the previous section. Then we have $n$ invertible Hecke roots $\lambda_1',\dots,\lambda_n'\in E^\times$. A consequence of relation \eqref{eq:contragredienthecke} is the fact that the map
$$
\lambda_{\mathfrak{p},i}\;\mapsto\;\lambda_{\mathfrak{p},n+1-i}^\vee:=\absNorm(\mathfrak{p})^n\lambda_{\mathfrak{p},i}^{-1}
$$
sets up a bijective correspondence between Hecke roots associated to
$$
v_{\mathfrak{p}}^0\;\in\;\mathscr{W}(\sigma_{\mathfrak{p}},\psi_\mathfrak{p}^{-1})
$$
and the Hecke roots associated to its dual vector
$$
(v_{\mathfrak{p}}^0)^\vee\;\in\;\mathscr{W}(\sigma_{\mathfrak{p}}^\vee,\psi_\mathfrak{p}),
$$
cf.\ Proposition 5.1. This relation is compatible with \eqref{eq:zentralscalar} in the sense that $\eta_n^\vee=\eta_n^{-1}$.

The same statements are true for $\pi_{\mathfrak{p}}$, if we also take into account the last omitted root $\lambda_{n+1}$. If our initial datum is of finite slope, then $\lambda_{n+1}\neq 0$ as the analog of relation \eqref{eq:zentralscalar} is valid as well. Therefore we may define
$$
\underline{\lambda}_{\mathfrak{p}}^\vee\;:=\;
(\lambda_{\mathfrak{p},1}^\vee,\dots,\lambda_{\mathfrak{p},n}^\vee,
{\lambda_{\mathfrak{p},1}'}^\vee,\dots,{\lambda_{\mathfrak{p},n}'}^\vee)
\in E^{2n},
$$
and the contragredient datum $(\pi^\vee,\sigma^\vee,(\underline{\lambda}_{\mathfrak{p}}^\vee)_{\mathfrak{p}\mid p})$ is of finite slope again, and we have the dual cohomology class
$$
\lambda^\vee
\;\in\;
H^{b_{n+1}^k+b_n^k}_{\rm c}(
\mathscr X_{n+1}^{\rm ad}(K)\times \mathscr X_{n}^{\rm ad}(K');
\underline{M}_{\boldsymbol{\mu}^\vee\times\boldsymbol{\nu}^\vee}(\BC)),
$$
which is again $p$-integral whenever $\lambda$ is, and is an eigen vector for $U_\mathfrak{p}$ with eigen value $\kappa_{\underline{\lambda}^\vee}\in E^\times$.

Now the map
$$
\cdot^\vee:\;\;\;\Gm(\BA_k^{(\infty)})\to\Gm(\BA_k^{(\infty)}),
$$
$$
x\;\;\mapsto\;\; x^\vee:=(-1)^{n}x^{-1},
$$
where the $(-1)^{n}$ occurs only in the $\mathfrak{p}$-components for $\mathfrak{p}\mid p$, induces an involution
$$
\cdot^\vee:C(K(p^\infty))\to C(K(p^\infty)),
$$
and also an involution $\cdot^\vee$ on $C(p^\infty)$, which commutes with the covering map \eqref{eq:Cfcover}.

\begin{thm}\label{thm:functionalequation}
We have the functional equation
$$
(\mu_\lambda(x))^\vee=
\mu_{\lambda^\vee}(x^{\vee}),
$$
i.e.\ we have explicitly
\begin{equation}
(\xi_j(\tilde{\mu}_\lambda(x)))^\vee=
\xi_{-j}(\tilde{\mu}_{\lambda^\vee}(x^{\vee})).
\label{eq:explicitfunctional}
\end{equation}
\end{thm}

\begin{proof}
The proof proceeds mutatis mutandis as the proof of Theorem 5.4 in \cite{Jan14}.
\end{proof}

\section{On the non-vanishing hypothesis}\label{sec:cohreps}

In this section we show how to reduce the non-vanishing of the archimedean periods $\Omega_t(\frac{1}{2}+j)$ in Theorem \ref{thm:interpolation} to the local situations of $\GL_{n+1}(k_v)\times\GL_n(k_v)$ for $v\in S_\infty$. Therefore its proof reduces to a real and a complex case, which is treated in \cite{Sun}.

We intrduce some notation, and let
$$
G:=G_{n+1}(\BR)\times G_n(\BR)
$$
and similarly
$$
GK:=GK_{n+1}\times GK_n,
$$
and
$$
K:=K_{n+1}\times K_n.
$$
We let furthermore
$$
H:= (j\times 1)(G_n(\BR))\subseteq G,
$$
$$
C:=H\cap GK,
$$
via the diagonal embedding \eqref{eq:diagonalembedding}. Then
$$
C\cong K_{n}.
$$
As usual we use the same notation for the Lie algebras.

We set
$$
M_{\boldsymbol{\mu}\times\boldsymbol{\nu}}
:=
M_{\boldsymbol{\mu}}(\BC)\otimes_\BC M_{\boldsymbol{\nu}}(\BC).
$$
The completed projective tensor product
$$
\widehat{\pi}:=\pi_{\infty}^{\rm CW}\widehat{\otimes}\sigma_{\infty}^{\rm CW}
$$
of the associated Casselman-Wallach representations is itself a Casselman-Wallach representation of $G$, which is unitarizable and tempered (modulo center), and satisfies
$$
H^{\bullet}(\g,GK;
\widehat{\pi}\otimes
M_{\boldsymbol{\mu}\times\boldsymbol{\nu}})
\neq 0.
$$
We will see that this already determines $\widehat{\pi}$ uniquely up to isomorphism, and we have
\begin{equation}
%H^{b_{n+1}^k+b_{n}^k}(\g_{n+1}\times\g_n,GK_{n+1}^0\times GK_n^0;
%\pi_\infty^{(K_{n+1})}\otimes M_{\boldsymbol{\mu}}(\BC)\otimes
%\sigma_\infty^{(K_{n})}\otimes M_{\boldsymbol{\nu}}(\BC))
H^{b_{n+1}^k+b_{n}^k}(\g,GK^0;
\widehat{\pi}\otimes
M_{\boldsymbol{\mu}\times\boldsymbol{\nu}})
\cong\BC[\pi_0(C)]
\label{eq:gkp0iso}
\end{equation}
as $\pi_0(C)$-modules. We will justify these statements below.

For each $j$ critical for ${\boldsymbol{\mu}\times\boldsymbol{\nu}}$ we have
$$
0\neq \xi_j\in\Hom_{H}(M_{\boldsymbol{\mu}\times\boldsymbol{\nu}},M_{(j)}(\BC))
$$
Since this is the same to say that $\frac{1}{2}+j$ is critical for $L(s,\pi\times\sigma)$, the archimedean Rankin-Selberg integrals computing $L(s,(\pi_\infty\times\sigma_\infty)\otimes\chi_\infty)$ are holomorphic at $\frac{1}{2}+j$, and produce non-zero continuous functionals
\begin{equation}\label{homiinfty}
  \phi_{\widehat{\pi}}^{\chi_\infty}\in \Hom_H(\widehat{\pi}, \chi_\infty\otimes\abs{\det}_\infty^{-j}).
\end{equation}
We remark that the finite order characters $\varepsilon$ of $H$ all factor over $\pi_0(H)=\pi_0(C)$ and thus may be considered as characters of the latter. Our goal is to prove

\begin{prp}\label{prop:periods}
For each character $\varepsilon$ of $\pi_0(H)$ the $H$-equivariant continuous linear functional
$$
\phi_{\widehat{\pi}}^\varepsilon\otimes\xi_j:\;\;\;
\widehat{\pi}\otimes
M_{\boldsymbol{\mu}\times\boldsymbol{\nu}}(\BC)
\to
\varepsilon\otimes\sgn^j
$$
induces on cohomology a non-zero $\pi_0(C)$-equivariant linear map
\begin{equation}\label{homhinfty}
   \oH^{b_{n+1}^{k}+b_{n}^{k}}(\g, GK^0; \widehat{\pi}\otimes M_{\boldsymbol{\mu}\times\boldsymbol{\nu}})\rightarrow
   \oH^{b_{n+1}^{k}+b_{n}^{k}}(\h, C^0; \varepsilon\otimes{\sgn}^j).
\end{equation}
\end{prp}

An easy calculation shows that
\[
\dim (\h/\c)=b_{n+1}^{k}+b_{n}^{k}.
\]
Therefore, Poincar\'{e} duality implies that the space of the right hand side of \eqref{homhinfty} is one-dimensional. It carries a representation of $\pi_0(C)$ which is isomorphic to $\varepsilon\otimes\sgn^j$. By \eqref{eq:gkp0iso} we know the structure of the left hand side as well. Thus Proposition \ref{prop:periods} provides us with a complete description of the Rankin-Selberg integrals on cohomology.

Proposition \ref{prop:periods} reduces to the local situation as follows. First we observe that
\begin{equation}
G\;=\;\prod_{v\in S_\infty}G_v
\label{eq:gdecomposition}
\end{equation}
with
$$
G_v=\GL_{n+1}(k_v)\times\GL_n(k_v),
$$
and similarly
$$
GK=\prod_{v\in S_\infty}GK_v.
$$
We have compatible decompositions
$$
H=\prod_{v\in S_\infty}H_v,\quad
H_v=\GL_n(k_v),
$$
and
$$
C=\prod_{v\in S_\infty} C_v,\quad
C_v\;=\;K_v,
$$
as in the notation section. We obtain induced decompositions of the associated complex Lie algebras, such as the Lie analog
$$
\g=\bigoplus_{v\in S_\infty} \g_v
$$
of \eqref{eq:gdecomposition}. Moreover
\begin{equation}
\pi_0(C)=\prod_{v\in S_\infty}\pi_0(C_v).
\label{eq:pi0cdecomposition}
\end{equation}
In particular any character $\varepsilon$ of $\pi_0(C)$ decomposes into a product of local characters $\varepsilon_v$, $v\mid\infty$. We have for each real place $v\in S_\infty$ a corresponding local sign character
\[
\sgn_v:=\det\,\abs{\det}_v^{-1},
\]
and defining $\sgn_v$
% and $\sgn(\nu_v)$ 
to be the trivial characters for complex places $v$, in particular
%the total sign
%\[
%\sgn({\bf\nu}):=\otimes_{v\mid\infty}\sgn_v,
%\]
%and
we get the obvious relation
\[
\sgn=\otimes_{v\mid\infty}\sgn_v.
\]

Furthemore, we have natural isomorphism
$$
\widehat{\pi}=\widehat{\otimes}_{v\in S_\infty}\widehat{\pi}_v,
\quad
\widehat{\pi}_v=\pi_v\widehat{\otimes}\sigma_v.
$$
Now by \cite[Lemma 3.14]{Clo}, we know that
$$
H^{b}(\g_v,GK_v^0; \widehat{\pi}_v\otimes M_{\mu_v\times\nu_v})=0
$$
for $b<b_{n+1}^{k_v}+b_n^{k_v}$. Therefore, the K\"unneth formula for relative Lie algebra cohomology provides us with an isomorphism
\begin{equation}
H^{b_{n+1}^k+b_{n}^k}(\g,GK^0;
\widehat{\pi}\otimes
M_{\boldsymbol{\mu}\times\boldsymbol{\nu}})
\cong
\bigoplus_{v\in S_\infty}
H^{b_{n+1}^{k_v}+b_{n}^{k_v}}(\g_v,GK_v^0;
\widehat{\pi}_v\otimes
M_{\mu_v\times\nu_v}).
\label{eq:kunneth}
\end{equation}

The local version of \eqref{eq:gkp0iso} reads
\begin{equation}
H^{b_{n+1}^{k_v}+b_{n}^{k_v}}(\g_v,GK_v^0;
\widehat{\pi}_v\otimes
M_{\mu_v\times\nu_v})
\cong\BC[\pi_0(C_v)]
\label{eq:gkp0isolocal}.
\end{equation}
In the sense of \eqref{eq:pi0cdecomposition}, the K\"unneth map \eqref{eq:kunneth} is an isomorphism of $\pi_0(C)$-modules, and is compatible with the isomorphisms \eqref{eq:gkp0iso} and \eqref{eq:gkp0isolocal}. Therefore \eqref{eq:gkp0isolocal} implies \eqref{eq:gkp0iso}.

Now let us justify the isomorphism \eqref{eq:gkp0isolocal}. For any place $v\in S_\infty$ we denote by $\Omega(\nu_v)$ the set of isomorphism classes of irreducible Casselman-Wallach representations $\rho_v$ of $\GL_{n}(k_v)$ such that $\rho_v|_{{\operatorname{SL}_n^{\pm}}(\k_v)}$ is unitarizable, tempered, and the relative Lie algebra cohomology
$$
\oH^\bullet(\g_v,GK_v^0;\rho_v\otimes M_{\nu_v})\neq 0
$$
is non-zero.

The set $\Omega(\nu_v)$ was determined for for complex $v$ by Enright in \cite{En} and for real $v$ by Speh in \cite{Sp} (see also \cite{VZ} for a general approach). In particular (cf.\ \cite[Section 3]{Clo}), for any place $v$
\begin{equation}\label{eq:omega}
  \#\Omega(\nu_v)=\left\{
                \begin{array}{ll}
                  0, & \hbox{if $\nu_v$ is not pure,} \\
                  1, & \hbox{if $\nu_v$ is pure, $v$ is real and $n$ is even or $v$ is complex and $n$ arbitrary} \\
                  2, & \hbox{if $\nu_v$ is pure, $v$ is real and $n$ is even.}
                \end{array}
              \right.
\end{equation}
Here \lq$\nu_v$ is pure\rq{} means (cf.\ \cite[Section 3]{Clo}),
\begin{equation}
  \nu_{\iota,1}+\nu_{\overline{\iota},n}=\nu_{\iota,2}+\nu_{\overline{\iota},n-1}=\cdots=\nu_{\iota,n}+\nu_{\overline{\iota},1},
\end{equation}
for $v=\{\iota,\overline{\iota}\}$ (a singleton whenever $v$ is real). This is the local equivalent to the prearithmeticity condition \eqref{eq:mduality} resp.\ \eqref{eq:mupurity} in section \ref{sec:arithmeticmodules}. In the third case of \eqref{eq:omega}, the two representations in $\Omega(\nu_v)$ are twists of each other by the sign character $\sgn_v$. In the second case of \eqref{eq:omega}, the only representation in $\Omega(\nu)$ is isomorphic to its twist by the sign character $\sgn_v$.

In the complex case the component group $\pi_0(C_v)$ is trivial and so is its action on the bottom degree $b^{k_v}$ of the relative Lie algebra cohomology of $\rho_v$, which is one-dimensional in this case. In the real case we have
$$
\BC[\pi_0(C_v)]={\bf 1}_v\oplus \sgn_v.
$$
An easy calculation shows that (cf.\ \cite[Equation (3.2)]{Mah} for example), as a representation of $\pi_0(C_v)$,
\begin{equation}\label{ohpi}
H^{b_n^{k_v}}(\g_v,GK_v^0;\rho_v\otimes M_{\nu_v})\cong \left\{
\begin{array}{ll}
{\bf1}_v\oplus \sgn_v, & \hbox{if $n$ is even,} \\
\sgn(\nu_v)\otimes\sgn_v^{\nu_{v,1}+\nu_{v,2}+\cdots +\nu_{v,n}}, & \hbox{if $n$ is odd.}
\end{array}
\right.
\end{equation}
Here $\sgn(\nu_v)$ is the character of $\pi_0(C_v)$ which is given by the action of the group
$$
\pm{\bf1}_{n}^{k_v}\;:=\;\pm{\bf1}_{n}\;\in\;\GL_n(k_v)\subseteq G_n(\BR)
$$
on the representation $M_{\nu_v}(\BC)$.
%At each real place such a matrix defines a section
%\begin{equation}
%\pi_0(C_v)\;\to\;C_v.
%\label{eq:pi0section}
%\end{equation}
%We use it to associate to the weight $\nu_v$ a sign $\sgn(\nu_v)$, which is the character of $\pi_0(C_v)=\pi_0(C_v)$ via which $\pi_0(C_v)$ acts on
%$$
%M_{\nu_v}:=M_{\nu_v}(\BC)
%$$
%via \eqref{eq:pi0section}.
%Dually
Therefore the slightly stricter condition \eqref{eq:pigkcoh} pins down $\widehat{\pi}$ uniquely. As already mentioned above, this also proves \eqref{eq:gkp0iso}.

Now observe that the archimedean Rankin-Selberg integrals
$$
\phi_{\widehat{\pi}}^\varepsilon(t_\infty)
=\int_{U_n(\BR)\backslash G_n(\BR)} t_\infty((j\times 1)(g))\varepsilon(\det(g))|\det(g)|_\infty^{s-\frac{1}{2}}dg
$$
naturally decompose into products of local integrals
\begin{equation}
\phi_{\widehat{\pi}_v}^\varepsilon(t_v)=
\int_{U_n(k_v)\backslash G_n(k_v)} t_v((j\times 1)(g))\varepsilon_v(\det(g))|\det(g)|_v^{s-\frac{1}{2}}dg,\;\;\;v\in S_\infty,
\label{eq:localrankin}
\end{equation}
whenever
$$
t_\infty\;\in\;
\mathscr W(\pi_\infty,\psi_\infty)\widehat{\otimes}
\mathscr W(\sigma_\infty,\psi_\infty^{-1})
$$
is a pure tensor with local components $t_v$.

By the K\"unneth isomorphism \eqref{eq:kunneth} our choice of cohomological test vector $t_\infty$ satisfies this condition. Therefore the non-vanishing of the complex periods $\Omega_t(\frac{1}{2}+j)$ from Theorem \ref{thm:interpolation} reduces to the question whether the local integrals \eqref{eq:localrankin} are non-zero for $s=\frac{1}{2}+j$. This proves Proposition \ref{prop:periods}.


\begin{thebibliography}{99}

\bibitem[BS]{borelserre1973}
A.~Borel and J.-P.~Serre, \emph{Corners and arithmetic groups}, Commentarii
  Mathematici Helvetici {\bf 48} (1973), 436-491.

\bibitem[Clo]{Clo}
L. Clozel, \textit{Motifs et formes automorphes: applications du
principe de fonctorialit\'{e}}.(French) [Motives and automorphic
forms: applications of the functoriality principle] Automorphic
forms, Shimura varieties, and L-functions, Vol.\ I (Ann Arbor, MI,
1988), 77-159, Perspect. Math., 10, Academic Press, Boston, MA,
1990.

\bibitem[CPR]{coatesperrinriou1989}
J. Coates, and B. Perrin-Riou, \emph{On $p$-adic $L$-functions
attached to motives over $\BQ$}. Algebraic number theory, Vol.\ 17,
Adv.\ Stud.\ Pure Math., Academic Press, 1989, pp.~23--54.

\bibitem[CPS]{cogdellpiatetskishapiro2004}
J.~W. Cogdell and I.~I. Piatetski-Shapiro, \emph{Remarks on
  {R}ankin-{S}elberg convolutions}, Contributions to
  automorphic forms, geometry and number theory ({H}. {H}ida, {D}.
  {R}amakrishnan, and {F}. {S}hahidi, eds.), John Hopkins University Press,
  2004, pp.~255--278.

\bibitem[Del]{Del}
P.~Deligne, \emph{Valeurs de fonctions {$L$} et p\'eriodes d'int\'egrales},
  Automorphic forms, representations and {$L$}-functions (Providence RI)
  (A.~Borel and W.~Casselman, eds.), Proceedings of the Symposium in Pure
  Mathematics {\bf 33}(2), American Mathematical Society, 1979, pp.~313--346.

\bibitem[En]{En}
T.~Enright, \textit{Relative Lie algebra cohomology and unitary
representations of complex Lie groups}, Duke Math.\ J.\ 46 (1979), 513-525.

\bibitem[HLTT]{HLTT2013}
M. Harris, K.-W. Lan, R. Taylor, and J. Thorne, \textit{On the
rigid cohomology of certain Shimura varieites}, preprint, 2013.

\bibitem[He36]{Hecke}
E. Hecke, \textit{\"Uber die Bestimmung Dirichletscher Reihen
  durch ihre Funktionalgleichung}, Mathematische Annalen {\bf 112} (1936), 664-699.

\bibitem[He37a]{Hecke2}
E. Hecke, \textit{Modulfunktionen und Dirichletschen Reihen
  mit Eulerscher Produktentwicklung I.}, Mathematische Annalen {\bf 114} (1937), 1-28.

\bibitem[He37b]{Hecke3}
E. Hecke, \textit{\"Uber Modulfunktionen und die Dirichletschen Reihen
  mit Eulerscher Produktentwicklung II.}, Mathematische Annalen {\bf 114} (1937), 316–351.


\bibitem[Jac]{Jac}
H. Jacquet, \textit{Archimedean Rankin-Selberg integrals}, in
\textit{Automorphic Forms and $L$-functions II: Local Aspects},
Proceedings of a workshop in honor of Steve Gelbart on the occasion
of his 60th birthday, Contemporary Mathematics, volumes 489,
57--172, AMS and BIU 2009.

\bibitem[JPSS]{jpss1983}
H.~Jacquet, I.~I. Piatetski-Shapiro, and J.~A. Shalika, \emph{{R}ankin-{S}elberg convolutions}, American Journal of
  Mathematics {\bf 105} (1983), 367--464.

\bibitem[Jan11]{Jan11}
F.~Januszewski, \emph{Modular symbols for reductive groups and $p$-adic
  {R}ankin-{S}elberg convolutions}, Journal f\"ur die reine und angewandte
  Mathematik {\bf 653} (2011), 1--45.

\bibitem[Jan14]{Jan14}
F. Januszewski, \textit{On p-adic L-functions for $\GL(n)\times  \GL(n-1)$ over totally real fields}, International Math.\ Res.\ Notices, 2014.

\bibitem[Jan15]{Janpre}
F. Januszewski, \textit{Cohomological induction over $\BQ$ and Frobenius-Schur indicators for $(\g,K)$-modules}, preprint.

\bibitem[KS]{KS}
H. Kasten and C.-G. Schmidt, \textit{On critical values of
Rankin-Selberg convolutions}, International Journal of Number Theory, 9 (2013), 205-256.

\bibitem[KMS]{KMS}
D. Kazhdan, B. Mazur and C.-G. Schmidt, \textit{Relative modular
symbols and Rankin-Selberg convolutions}, J. reine angew. Math. 519
(2000), 97-141.

\bibitem[Mah]{Mah}
J. Mahnkopf, \textit{Cohomology of arithmetic groups, parabolic
subgroups and the special values of L-functions of $\GL_n$}, J.
Inst. Math. Jussieu., 4(4), 553-637 (2005).

\bibitem[PR]{perrinriou1995}
B. Perrin-Riou, \emph{Fonctions $L$ $p$-adiques des repr\'esentations $p$-adiques.} Ast\'erisque 229, SMF, 1995.

\bibitem[Rag2]{Ragpre}
A. Raghuram, \textit{Critical values of {R}anking-{S}elberg ${L}$-functions for ${\rm GL}_n\times{\rm GL_{n-1}}$ and the symmetric cube ${L}$-functions for ${\rm GL}_2$}, arXiv:1312.5955. 

%\bibitem[RS1]{RS1}
%A. Raghuram and F. Shahidi, \textit{Functoriality and special values of
%L-functions}, in: \textit{Eisenstein series and Applications}, eds
%W. T. Gan, S. Kudla, Y. Tschinkel, Progress in Mathematics 258
%(Boston, 2008), pp. 271-294.

\bibitem[Sch]{Sch}
C.-G.~Schmidt, \textit{Relative modular symbols and $p$-adic
Rankin-Selberg convolutions}, Invent. Math. 112 (1993), 31-76.

\bibitem[Sch2]{Sch2}
C.-G.~Schmidt, \textit{Period relations and $p$-adic measures},
manuscr.\ math.\ 106 (2001), 177-201.

\bibitem[Schl]{scholze2013}
P.~Scholze, \textit{On torsion in the cohomology of locally symmetric
varieties}, preprint, 2013.

\bibitem[Sp]{Sp}
B.~Speh, \textit{Unitary representations of ${\rm GL}(n,\R)$ with non-trivial $(\g,K)$-cohomology}, Invent.\ Math.\ 71 (1983), 443-465.

\bibitem[Sun]{Sun}
B.~Sun, \textit{The non-vanishing hypothesis at infinity for Rankin-Selberg convolutions}, preprint.

\bibitem[VZ]{VZ}
D. Vogan and G. Zuckerman, \textit{Unitary representations with
non-zero cohomology}, Compositio Math. 53 (1984)), 51-90.


\end{thebibliography}
\end{document}